\documentclass[11pt]{article}
\usepackage{amsfonts,amssymb,amsmath}
\usepackage{graphicx}

\usepackage{color} 
   \definecolor{cites}{rgb}{0.50 , 0.00 , 0.00}  
   \definecolor{urls} {rgb}{0.00 , 0.00 , 0.50}  
   \definecolor{links}{rgb}{0.00 , 0.00 , 0.50}   
\usepackage[
      colorlinks=true,   
      citecolor=cites,   
      urlcolor=urls,     
      linkcolor=links,   
      pdftitle={Finite sections of random Jacobi operators},  
      pdfauthor={Marko Lindner, Steffen Roch}, 
      pdfpagemode=UseOutlines, 
      pdfstartview=FitH,       
      bookmarksopen=false      
   ]{hyperref}

\newcommand{\cA}{{\mathcal A}}

\newcommand\eps\varepsilon
\newcommand\ph\varphi
\newcommand\spec{{\rm spec}\,}  
\newcommand\specn{{\rm spec}}   

\newcommand\spess{{\rm spec}_{\rm ess}\,}
\newcommand\spptp{{\rm spec}_{\rm pt}^p\,}
\newcommand\sppt{{\rm spec}_{\rm pt}^\infty\,}
\newcommand\opsp{\sigma^{\sf op}}
\newcommand\dist{{\rm dist}}

\newcommand\diag{{\rm diag}}
\newcommand\Diag{{\rm Diag}}
\newcommand\conv{{\rm conv}}

\newcommand\ri{{\rm i}}
\newcommand\ind{{\rm ind\,}}
\newcommand\wind{{\rm wind}}
\newcommand\im{{\rm im\,}}

\newcommand\C{{\mathbb C}}
\newcommand\R{{\mathbb R}}

\newcommand\T{{\mathbb T}}
\newcommand\Z{{\mathbb Z}}
\newcommand\N{{\mathbb N}}
\newcommand\I{{\mathbb I}}
\newcommand\M{{\mathcal M}}
\newcommand\D{{\mathbb D}}
\newcommand\ovD{{\overline{\mathbb D}}}

\newcommand\uu{{\mathbf u}}
\newcommand\vv{{\mathbf v}}
\newcommand\ww{{\mathbf w}}
\newcommand\PE{{\rm \Psi E}} 
\newcommand\PEL{\PE_{\rm L}}
\newcommand\PER{\PE_{\rm R}}

\newtheorem{theorem}{Theorem}[section]
\newtheorem{lemma}[theorem]{Lemma}
\newtheorem{corollary}[theorem]{Corollary}
\newtheorem{proposition}[theorem]{Proposition}

\newtheorem{algorithm}[theorem]{Algorithm}

\newenvironment{remark}
 {\par\noindent\refstepcounter{theorem}{\bf Remark \thetheorem}\ }
 {\raisebox{1mm}{\framebox{}}\pagebreak[2]}

\newenvironment{example}
 {\par\noindent\refstepcounter{theorem}{\bf Example \thetheorem}\ }
 {\raisebox{1mm}{\framebox{}}\pagebreak[2]}

\newenvironment{proof}
 {\par\noindent{\bf Proof.}}
 {\rule{2mm}{2mm}\pagebreak[2]}

\newenvironment{proofof}[1]
 {\par\noindent{\bf Proof of #1.}}
 {\rule{2mm}{2mm}\pagebreak[2]}

\numberwithin{figure}{section}  

\begin{document}
\title{\bf Finite sections of random Jacobi operators}
\author{{\sc Marko Lindner}\quad and\quad {\sc Steffen Roch}}
\date{\today}
\maketitle
\begin{quote}
\renewcommand{\baselinestretch}{1.0}
\footnotesize {\sc Abstract.} This article is about a problem in the
numerical analysis of random operators. We study a version of the
finite section method for the approximate solution of equations
$Ax=b$ in infinitely many variables, where $A$ is a random Jacobi
operator. In other words, we approximately solve infinite second
order difference equations with stochastic coefficients by reducing
the infinite volume case to the (large) finite volume case via a
particular truncation technique. For most of the paper we consider
non-selfadjoint operators $A$ but we also comment on the
self-adjoint case when simplifications occur.
\end{quote}

\section{Introduction}
Let $U,\,V$ and $W$ be non-empty compact subsets of the complex
plane, and put
\begin{equation} \label{eq:defminmax}
u^*\,:=\,\max_{u\in U}|u|,\quad v_*\,:=\,\min_{v\in V}|v|,\quad
w^*\,:=\,\max_{w\in
W}|w|\quad\textrm{and}\quad\delta\,:=\,v_*-(u^*+w^*).
\end{equation}
We write $\N$, $\Z$, $\R$ and $\C$ for the sets of all
positive integer, integer, real and complex numbers.
\medskip

{\bf Infinite matrices. } In this paper we study bi- and
semi-infinite matrices of the form
\begin{equation} \label{eq:A}
A = \left(\begin{array}{ccccccc} \ddots&\ddots\\
\ddots&v_{-2}&w_{-2}\\
&u_{-1}&v_{-1}&w_{-1}\\
\cline{4-4}
&&u_{0}&\multicolumn{1}{|c|}{v_0}&w_0\\
\cline{4-4}
&&&u_{1}&v_1&w_1\\
&&&&u_2&v_2&\ddots\\
&&&&&\ddots&\ddots
\end{array}\right)\quad \textrm{and}\quad
A_+ = \left(\begin{array}{ccccc}
v_{1}&w_{1}\\
u_{2}&v_{2}&w_{2}\\
&u_{3}&v_3&w_3\\
&&u_{4}&v_4&\ddots\\
&&&\ddots&\ddots
\end{array}\right)
\end{equation}
with entries $u_i\in U$, $v_i\in V$ and $w_i\in W$ for all $i$ under
consideration, where the box marks the matrix entry of $A$ at
$(0,0)$. As usual, we call $\uu:=(u_i)$, $\vv:=(v_i)$, and
$\ww:=(w_i)$ the sub-, main- and superdiagonal of $A$, resp.\!
$A_+$. We understand $A$ and $A_+$ as linear operators, again
denoted by $A$ and $A_+$, acting boundedly, by matrix-vector
multiplication, on the standard spaces $\ell^p(\Z)$ and $\ell^p(\N)$
of bi- and semi-infinite complex sequences with $p\in [1,\infty]$.
It is clear that the matrices \eqref{eq:A} are in general not
self-adjoint. We will study the selfadjoint case, when
$w_i=\overline u_{i+1}$ for all $i$, separately in Section
\ref{sec:selfadjoint}.

The sets of all operators $A$ and $A_+$ from \eqref{eq:A} with
entries $u_i\in U$, $v_i\in V$ and $w_i\in W$ for all indices $i$
that occur will be denoted by $M(U,V,W)$ and $M_+(U,V,W)$,
respectively. The set of all $n\times n$ matrices with subdiagonal
entries in $U$, main diagonal entries in $V$ and superdiagonal
entries in $W$ (and all other entries zero) will be called
$M_n(U,V,W)$ for $n\in\N$, and we finally put $M_{\rm
fin}(U,V,W)=\cup_{n\in\N}\, M_n(U,V,W)$.

Recall that a bounded linear operator $B:X\to Y$ between Banach
spaces is a {\sl Fredholm operator} if the dimension, $\alpha$, of
its null-space is finite and the codimension, $\beta$, of its image
in $Y$ is finite. In this case, the image of $A$ is
closed in $Y$ and the integer $\ind A:=\alpha-\beta$ is called the
{\sl index} of $A$. For a bounded linear operator $B$ on
$\ell^p(\I)$ with $\I\in\{\Z,\N,\Z\setminus\N\}$, we write $\specn^p
B$, $\specn_{\rm ess}^p\,B$ and $\spptp B$ for the sets of all
$\lambda\in\C$ for which $B-\lambda I$ is, respectively, not
invertible, not a Fredholm operator or not injective on
$\ell^p(\I)$. Because $A$ and $A_+$ in \eqref{eq:A} are band
matrices, their spectrum and essential spectrum do not depend on the
underlying $\ell^p$-space \cite{Kurbatov,Li:Wiener,Roch:ellp}, so
that we will just write $\spec B$ and $\spess B$ for operators $B$
in $M(U,V,W)$ and in $M_+(U,V,W)$.
\medskip

{\bf Random alias pseudoergodic operators. } Our particular interest
is on random operators in $M(U,V,W)$ and $M_+(U,V,W)$. We model
randomness by the following concept: Given a metric space $(\M,d)$
and an index set $\I\in\{\Z,\N,\Z\setminus\N\}$, we say that a
sequence $a=(a_i)_{i\in\I}$ in $\M$ is {\sl pseudoergodic} if for
every $\eps>0$, all $n\in\N$ and all $b=(b_i)_{i=1}^n\in\M^n$, there
is a $k\in\I$ such that $d(a_{k+i},b_i)<\eps$ for all $i=1,...,n$.
In particular, if $(\M,d)$ is a discrete space then $a=(a_i)$ is
pseudoergodic if and only if every finite vector over $\M$ can be
found (as a sequence of consecutive entries) in $a$. For a finite
set $\M$, a pseudoergodic sequence $(a_i)_{i\in\N}$ can be
constructed by writing all $\M$-valued sequences of length $1$, then
$2$, then $3$, $\dots$ in a row. For $\M=\{0,1\}$, this is done by
stringing together the binary expansions of all natural numbers.

We will call $(a_i)_{i\in\Z}$ : {\sl left-pseudoergodic} if
$(a_i)_{i\in\Z\setminus\N}$ is pseudoergodic, {\sl
right-pseudoergodic} if $(a_i)_{i\in\N}$ is pseudoergodic, and {\sl
bi-pseudoergodic} if both $(a_i)_{i\in\Z\setminus\N}$ and
$(a_i)_{i\in\N}$ are pseudoergodic. It is a simple exercise to show
that $(a_i)_{i\in\Z}$ is pseudoergodic if and only if it is right-
or left-pseudoergodic.

Pseudoergodicity was introduced by Davies
\cite{Davies2001:PseudoErg} to study spectral properties of random
operators while eliminating probabilistic arguments. Indeed, if
$a=(a_i)_{i\in\I}$, where all entries $a_i$ are independent (or at
least not fully correlated) samples from a random variable with
values (densely) in $\M$ then, with probability one, $a$ is
pseudoergodic in cases $\I\in\{\N,\Z\setminus\N\}$ and
bi-pseudoergodic in case $\I=\Z$ (e.g. \cite[\S 5.5.3]{Li:Habil}).

We call an operator $A\in M(U,V,W)$ {\sl pseudoergodic}, {\sl
left-pseudoergodic}, {\sl right-pseudoergodic} or {\sl
bi-pseudo\-ergodic} and write $A\in\PE(U,V,W)$, $A\in\PEL(U,V,W)$,
$A\in\PER(U,V,W)$ or $A\in\PE_2(U,V,W)$, respectively, if
$a=(a_i)_{i\in\Z}$ with $a_i:=(u_i,v_i,w_i)\in \M:= U\times V\times
W\subset\C^3$ has the corresponding property. So we have
\begin{eqnarray*}
\PE(U,V,W) &=& \PEL(U,V,W)\ \cup\ \PER(U,V,W),\\
\PE_2(U,V,W) &=& \PEL(U,V,W)\ \cap\ \PER(U,V,W).
\end{eqnarray*}
If $A\in\PER(U,V,W)$ then we will write $A_+\in\PE_+(U,V,W)$ for the
corresponding semi-infinite submatrix $A_+$ of $A$ from
\eqref{eq:A}.
We will say a little bit about spectral properties of pseudoergodic
operators $A$ and $A_+$ but will mainly focus on another problem:
\medskip

{\bf The finite section method (FSM). } If one wants to solve an
equation
\begin{equation} \label{eq:Ax=b}
Ax\,=\,b,\qquad\textrm{i.e.}\qquad \sum_{j\in\Z} a_{ij}\ x(j)\ =\
b(i),\quad i\in\Z
\end{equation}
on $X=\ell^p(\Z)$ approximately, where $A:X\to X$ (bounded) and
$b\in X$ are given and $x\in X$ is sought for, one often uses a
projection method. Therefore, let $P_{l,r}:X\to X$ stand for the
operator of multiplication by the characteristic function of the
discrete interval $\Z \cap [l,...,r]$ for $l,r\in\Z$ with $l\le r$
and denote the image of
$P_{l,r}$ by $X_{l,r}\cong \C^{r-l+1}$. One then picks sequences of
integers $l_1,l_2,...\to -\infty$ and $r_1,r_2,...\to+\infty$ and
replaces the infinite system \eqref{eq:Ax=b} by the sequence of
finite systems
\begin{equation} \label{eq:Anxn=b}
P_{l_n,r_n}AP_{l_n,r_n}x_n\,=\,P_{l_n,r_n}b,\qquad\textrm{i.e.}\qquad
\sum_{l_n\le j\le r_n} a_{ij}\ x_n(j)\ =\ b(i),\quad l_n\le i\le r_n
\end{equation}
with $n \in \N$ . The aim is that, assuming invertibility of $A$
(i.e. unique solvability of \eqref{eq:Ax=b} for all $b\in X$), also
\eqref{eq:Anxn=b} shall be uniquely solvable for all sufficiently
large $n$ and the solutions $x_n\in X_{l_n,r_n}$ shall remain
bounded in $n$ and converge componentwise\footnote{For
$p\in(1,\infty)$ this is equivalent \cite{RaRoSiBook} to convergence
of the solutions $x_n$ (extended by zero) to $x$ in $X=\ell^p(\Z)$.}
to the solution $x$ of \eqref{eq:Ax=b}. If the latter is the case
for all right-hand sides $b\in X$ then we say that the {\sl finite
section method (short: FSM)} with cut-offs at $(l_n)$ and $(r_n)$
{\sl is applicable} for $A$.

If $(a_{ij})$ is a band matrix, i.e. $a_{ij}=0$ for $|i-j|>d$ with
some $d\in\N$ (as is the case for our operators \eqref{eq:A}), then
the FSM is applicable if and only if $A$ is invertible and the
sequence $(A_n):=(P_{l_n,r_n}AP_{l_n,r_n})_{n\in\N}$ is {\sl stable}
\cite{RoSi}. By the latter we mean that there exists a $n_0\in\N$
such that $A_n:X_{l_n,r_n}\to X_{l_n,r_n}$ is invertible for all
$n\ge n_0$ and $\sup_{n\ge n_0} \|A_n^{-1}\|<\infty$.

In the case $l_n=-n,\ r_n=n$ we will speak of the {\sl full FSM} for
$A$. Recently, it has been shown in different situations that (and
how) applicability of the FSM can be established by choosing the
sequences $(l_n)$ and $(r_n)$ accordingly
\cite{Li:FSMsubs,RaRoSi:FSM_AP,RaRoSi:FSMsubs} if the full FSM is
not applicable. We will formulate a condition for applicability of
the FSM in the case of general sequences $(l_n)$ and $(r_n)$ and
then apply this result to the case of pseudoergodic operators
\eqref{eq:A}.
For semi-infinite systems on $X=\ell^p(\N)$, replace $\Z$ by $\N$
and $l_n$ by $1$ in all of the above.
\medskip

{\bf Motivation. } A major motivation for the study of random Jacobi
operators, their spectra and the solutions of the corresponding
operator equations comes from condensed matter physics: Questions
about the conductivity of certain (composed, disordered) media,
about flux lines in superconductors or about systems of asymmetricly
hopping particles have been modeled by random Schr\"odinger
operators (Anderson model \cite{Anderson58,Anderson61}),
non-selfadjoint versions (Hatano \& Nelson
\cite{HatanoNelson96,HatanoNelson97,HatanoNelson98}) and other
non-selfadjoint random Jacobi operators (Feinberg \& Zee
\cite{FeinZee99a,FeinZee99b}). Similar models arise in population
biology \cite{NelsonShnerb}. Besides such discrete models also
continuous problems that have been described by a stochastic
differential equation in 1D lead, after suitable discretization, to
a matrix equation of the kind studied here.

We give some upper and lower bounds on the spectrum of our operators
but mainly focus on the approximate solution of operator equations
$Ax=b$ via the FSM. The latter can however be useful for spectral
studies again: The inverse power method for the computation of the
eigenvalue of $A$ that is closest to a given point $z\in\C$
approximates the (in modulus) largest eigenvalue of $(A-zI)^{-1}$ by
repeatedly solving equations $(A-zI)x^{(n+1)}=x^{(n)}$, $n=0,1,...$,
with a rather arbitrary (non-zero) initial vector $x^{(0)}$.

%
%
%
\medskip

{\bf Historic remarks. } The idea of the FSM is so natural that it
is difficult to give a historical starting point. First rigorous
treatments are from Baxter \cite{Baxter} and Gohberg \& Feldman
\cite{GohbergFeldman} on Wiener-Hopf and convolution operators in
dimension $N=1$ in the early 1960's. For convolution equations in
higher dimensions $N\ge 2$, the FSM goes back to Kozak \& Simonenko
\cite{Kozak,KozakSimonenko}, and for general band-dominated
operators with scalar \cite{RaRoSi1998} and operator-valued
\cite{RaRoSi2001,RaRoSi2001:FSM} coefficients, most results are due
to Rabinovich, Roch \& Silbermann. For the state of the art in the
scalar case for $p=2$, see \cite{Roch:FSM}. The quest for stable
subsequences if the full FSM itself is instable is getting more
attention recently
\cite{RaRoSi:FSM_AP,RaRoSi:FSMsubs,SeidelSilbermann1,SeidelSilbermann2,Li:FSMsubs}.
In \cite{RaRoSi:FSMsubs}, the stability theorem for subsequences is
used to simplify the criterion in dimension $N=1$ by removing a
uniform boundedness condition. However, we are not aware of a
rigorous treatment of random (or pseudoergodic) operators via the
finite section method.


\section{Main results}
\subsection{Notations}
We first need some geometric notations: For sets $S,T\subseteq\C$ we
put $S+T:=\{s+t:s\in S,\,t\in T\}$ and we write $s+T:=\{s\}+T$ and
$sT:=\{st:t\in T\}$ if $s\in\C$. By $\T=\{z\in\C:|z|=1\}$,
$\D=\{z\in\C:|z|<1\}$ and $\ovD=\D\cup\T$ we denote the unit circle,
the unit disk and its closure. So, for example, $S+\eps\ovD$ is the
closed $\eps$-neighborhood of $S\subseteq\C$ with $\eps> 0$.

\noindent
\begin{center}
\includegraphics[width=\textwidth]{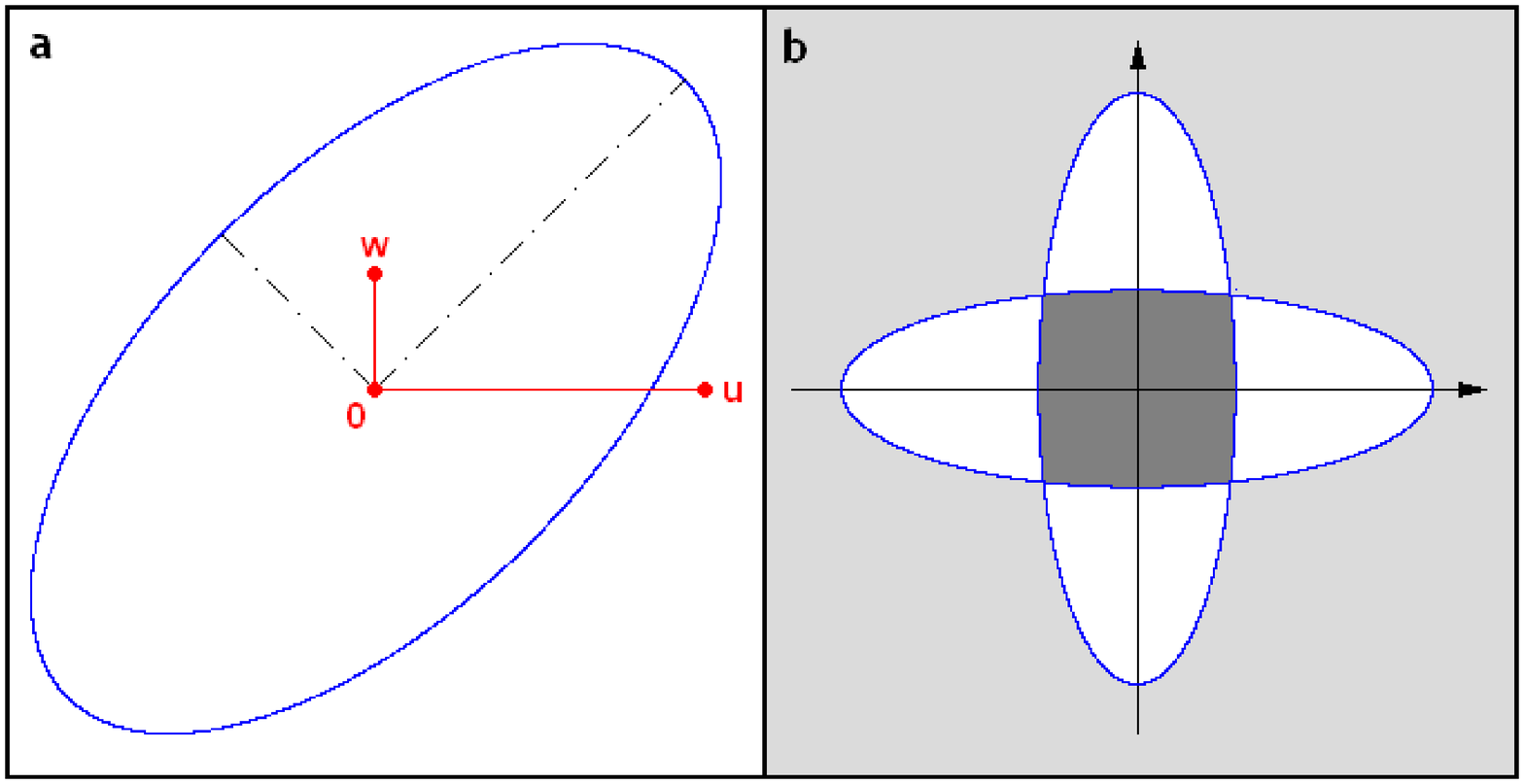}
\end{center}
~\\[-17mm]  
\begin{figure}[h]
\caption{\footnotesize a) This is the ellipse $E(u,w)$ with $u=3$
and $w=\ri$. The major axis of the ellipse bisects the angle between
$u$ and $w$ at the origin. The half-axes (dotted lines) have length
$|3|\pm|\ri|$, i.e. $4$ and $2$.\newline b) We see $E_+(U,W)$ in
dark gray and $E_-(U,W)$ in light gray for $U=\{2\}$ and
$W=\{-1,1\}$. $E(U,W)$ is the union of the two ellipses $E(2,-1)$
and $E(2,1)$. In Sections \ref{sec:spec} and \ref{sec:Fred} we show
that, for $A\in\PE(U,V,W)$ to be Fredholm (and hence invertible) it
is necessary that either $V\subseteq E_+(U,W)$, in which case $\ind
A_+=-1$, or $V\subseteq E_-(U,W)$, in which case $\ind A_+=0$. On
the other hand, if $V\subset\D\subset E_+(U,W)$ then $A$ is
invertible.} \label{fig:ellipse}
\end{figure}

For $u,w\in\C$, put
\begin{equation} \label{eq:ellipse}
E(u,w)\ :=\ \big\{\, v\in\C\ :\
|v+2\sqrt{uw}|+|v-2\sqrt{uw}|\,=\,2(|u|+|w|)\,\big\},
\end{equation}
which is the ellipse that is centered at 0, has half-axes of length
$|u|+|w|$ and $\big|\,|u|-|w|\,\big|$ and focal points $\pm
2\sqrt{uw}$ (so that the major axis of $E(u,w)$ bisects the angle
between $u$ and $w$ at the origin).
By $E_+(u,w)$ and $E_-(u,w)$ we denote the bounded (interior) and
the unbounded (exterior) component of $\C\setminus E(u,w)$,
respectively. Now, for non-empty $U,W\subset\C$, let
\begin{equation*} \label{eq:ellsets}
E(U,W)\ :=\ \bigcup_{\scriptsize \begin{array}{c}u\in U\\[-0.5mm]w\in
W\end{array}} E(u,w)\qquad\textrm{and}\qquad E_\pm(U,W)\ :=\
\bigcap_{\scriptsize \begin{array}{c}u\in U\\[-0.5mm]w\in
W\end{array}} E_\pm(u,w).
\end{equation*}
Note that $E(u,w)=-E(u,w)$ for all $u,w\in\C$, so that also
$E(U,W)=-E(U,W)$ and $E_\pm(U,W)=-E_\pm(U,W)$ hold (see Figure
\ref{fig:ellipse} for an example).

\subsection{Spectrum and essential spectrum} \label{sec:spec}
Let $U,V,W\subset\C$ be non-empty and compact sets, and recall
\eqref{eq:defminmax}. Then we have the following result about
spectrum and essential spectrum of our pseudoergodic operators
\eqref{eq:A}:

\begin{theorem} \label{th:spec}
{\bf a) } For $A\in\PE(U,V,W)$ and $A_+\in\PE_+(U,V,W)$, the
following holds:
\begin{eqnarray} \nonumber
V+E(U,W) &\subseteq& \\
\nonumber \bigcup_{B\in M(U,V,W)}\!\!\!\!\!\!\!\!\!\spec B& =&
\bigcup_{B\in M(U,V,W)}\!\!\!\!\!\!\!\!\!\sppt B \ =\ \bigcup_{B\in
M(U,V,W)}\!\!\!\!\!\!\!\!\!\spess B
\ =\ \bigcup_{B_+\in M_+(U,V,W)}\!\!\!\!\!\!\!\!\!\!\!\spess B_+\\
\label{eq:specA} &=& \spess A\ =\ \spec A\ =\ \spess A_+\ \subseteq\
\spec
A_+\\
\nonumber &\subseteq& V+(u^*+w^*)\ovD
\end{eqnarray}

{\bf b) } The upper bound $\spec A \subseteq V+(u^*+w^*)\ovD$ from
\eqref{eq:specA} holds for arbitrary $A\in M(U,V,W)$ (as well as for
semi-infinite and finite matrices $A$). For bi-infinite matrices
$A\in M(U,V,W)$ we can improve this upper bound on $\spec A$ under
one of the following conditions:

\begin{tabular}{l}
if $u_*>w^*$, i.e. $|u|>|w|$ $\forall u\in U, w\in W$, then
$\displaystyle \spec
A\subset\C\setminus\bigcap_{v\in V}\Big(v+(u_*-w^*)\D\Big)$,\\
if $u^*<w_*$, i.e. $|u|<|w|$ $\forall u\in U, w\in W$, then
$\displaystyle\spec A\subset\C\setminus \bigcap_{v\in
V}\Big(v+(w_*-u^*)\D\Big)$,
\end{tabular}\\
where, in addition to \eqref{eq:defminmax}, we define $u_* :=
\min_{u\in U}|u|$ and $w_* := \min_{w\in W}|w|$.
\end{theorem}

%
%

We see from \eqref{eq:specA} that the spectrum of $A$ and the
essential spectrum of $A$ and $A_+$ only depend on the sets $U,V,W$
but not on the pseudoergodic operators $A$ and $A_+$. So all
operators in $\PE(U,V,W)$ have the same (essential) spectrum. In
particular, in the case of random operators, $\spess A$, $\spec A$
and $\spess A_+$ do not depend on the distributions of the random
variables for sub-, main- and superdiagonal -- only on their
supports $U,V,W$. (Note that none of the above applies to $\spec
A_+$; this set does depend on the concrete operator $A_+$.)

\begin{example} \label{ex:spec}
{\bf a) Anderson model. } In \cite{Anderson58,Anderson61}, the
conductivity of 1D disordered media was studied. Here $U=W=\{1\}$
and $V\subset\R$, so that $A$ is a discrete Schr\"odinger operator
with random potential. In this case $E(U,W)=[-2,2]$. So our lower
and upper bound from Theorem \ref{th:spec} are $V+[-2,2]$ and
$V+2\ovD$. Together with $\spec A\subset\R$, by selfadjointness, we
get that the lower bound is also an upper bound, whence all sets in
\eqref{eq:specA} are equal to $V+[-2,2]$.

{\bf b) Hatano \& Nelson. } The so-called non-selfadjoint Anderson
model was introduced in
\cite{HatanoNelson96,HatanoNelson97,HatanoNelson98} for the study of
flux lines in type II superconductors under the influence of a
tilted external magnetic field. Here $U=\{e^g\}$, $V=[-a,a]$ and
$W=\{e^{-g}\}$, where $a$ and $g$ (the strength of the magnetic
field) are positive real parameters. Now $E:=E(U,W)$ has half-axes
of length $e^g\pm e^{-g}$ being part of the real and imaginary axis,
and $E$ gets closer to a circle as $g\to\infty$. Abbreviate $e^g+
e^{-g}=2 \cosh g=:c$ and $e^g- e^{-g}=2 \sinh g=:s$. If $V$ is (at
least) as long as the major axis of $E$, i.e. if $a\ge c$, then
$V+E$ and $V+c\ovD$, which are the lower and upper bound in
\eqref{eq:specA}, only differ by $2e^{-g}$ in Hausdorff distance. It
is easy to see that the closed numerical range of $A$ (which is
always an upper bound on $\spec A$) is contained in $V+\conv(E)$,
which is equal to $V+E$ if and only if $a\ge c$, so that all sets in
\eqref{eq:specA} are equal to $V+E$ in this case (cf.
\cite{Davies2001:SpecNSA}). If $a<c$ then the lower bound $V+E$ has
a hole around the origin, so that the best we can say then is
$V+E\subseteq\spec A\subseteq V+\conv(E)$. However, if $V$ is
shorter than the short axis of $E$, i.e. if $a<s$, then statement b)
of the theorem proves that there is indeed a hole in $\spec A$:
Since $u_*=e^g>e^{-g}=w^*$, we have that $\cap_{v\in
V}(v+(u_*-w^*)\D)=(-a+s\D)\cap(a+s\D)\ne\varnothing$ is in the
resolvent set of $A$. We summarize these bounds on $\spec A$ in case
$a<s$ in Figure \ref{fig:HatNel} below. A further study of the shape
and size of the hole in $\spec A$ is in
\cite{Davies2001:SpecNSA,Davies2001:PseudoErg,Davies2005:HigherNumRanges,MartinezThesis,MartinezHN}.

{\bf c) Feinberg \& Zee. } In
\cite{FeinZee99a,FeinZee99b,HolzOrlandZee} the case $U=\{1\}$,
$V=\{0\}$, $W=\T$ is studied. A simple computation shows that then
$E(U,W)=2\ovD=(u^*+w^*)\ovD$ holds, so that all sets in
\eqref{eq:specA} coincide. In the same papers the much more
complicated case with $W=\{\pm1\}$ is also studied. In this case,
$E(U,W)=[-2,2]\cup[-2\ri,2\ri]$ is far away from
$(u^*+w^*)\ovD=2\ovD$ (see
\cite{CW.Heng.ML:Sierp,CW.Heng.ML:UpperBounds,CW.Heng.ML:tridiag,HolzOrlandZee}
for sharper bounds in this case).
\end{example}

\begin{remark}
The upper and lower bound in \eqref{eq:specA} might create the
impression that spectrum and essential spectrum of $A$ can be
written as $V+S(U,W)$ with a set $S(U,W)$ independent of $V$, in
which case it would be sufficient to study the case $V=\{0\}$. To
see that this is not true, compare the cases $U\times V\times
W=\{0\}\times\{\pm 1\}\times \{1\}$ (see
\cite{TrefContEmb,Li:BiDiag}) and $U\times V\times
W=\{0\}\times\{0\}\times \{1\}$: In the first case one has
$S(U,W)=\ovD$, whereas in the second case, $S(U,W)=\T$.
\end{remark}

Both upper and lower bound in Theorem \ref{th:spec} can be improved:
The lower bound comes from evaluating $\cup\,\spec B$ with the union
taken over all $B\in M(U,V,W)$ that have constant diagonals. A
better lower bound can be derived in concrete examples by also
considering operators $B\in M(U,V,W)$ with diagonals of period 2, 3
or more (e.g.
\cite{CW.Heng.ML:Sierp,CW.Heng.ML:tridiag,Davies2001:SpecNSA,Li:BiDiag}).
The upper bound can be improved by different approaches such as
(higher order) numerical ranges or hulls
\cite{Davies2005:HigherNumRanges,Davies2007:Book} or by the more
recent ideas of \cite{CW.Heng.ML:UpperBounds}.

\subsection{Fredholmness, invertibility and the full FSM}
\label{sec:Fred}
Here are our results on invertibility, Fredholm property, and applicability of the full FSM.
\begin{theorem} \label{th:Fred}
If $A\in\PE(U,V,W)$ or $A_+\in\PE_+(U,V,W)$ is Fredholm then $V\cap
E(U,W)=\varnothing$. In fact, either

\begin{tabular}{cl}
(a)&$V\subseteq E_-(U,W)$,\quad or\\
(b)&$V\subseteq E_+(U,W)$ and $u_*>w^*$,\quad or\\
(c)&$V\subseteq E_+(U,W)$ and $u^*<w_*$,
\end{tabular}

\noindent where, in addition to \eqref{eq:defminmax}, we define
$u_* := \min_{u\in U}|u|$ and $w_* := \min_{w\in W}|w|$.

The three cases correspond to the Fredholm index of $A_+$ (the
so-called {\sl plus-index} of $A$):\\ (a)$\iff \ind A_+=0$,\quad
(b)$\iff \ind A_+ = -1$\quad and\quad (c)$\iff \ind A_+=1$.
\end{theorem}

Note that, while the index of $A_+$ can be $-1$, $0$ or $1$, the
index of $A$ is always zero if $A$ is Fredholm; in fact, $A$ is
always invertible if Fredholm (see \eqref{eq:specA} or the following
theorem).

%

\newpage 

\begin{theorem} \label{th:Fred_inv_stab}
Let $U,V,W\subset\C$ be non-empty and compact. For $A\in M(U,V,W)$,
we look at the following statements:

\begin{tabular}{rl}
  (i)&$A$ is a Fredholm operator,\\
 (ii)&$A$ is invertible,\\
(iii)&the full FSM is applicable to $A$,\\
 (iv)&all operators in $M(U,V,W)$, $M_+(U,V,W)$ and $M_{\rm fin}(U,V,W)$ are invertible\\
     &and their inverses are uniformly bounded from above,\\
  (v)&the full FSM is applicable to all operators in $M(U,V,W)$,\\
 (vi)&all $B\in M(U,V,W)$ are invertible,\\
(vii)&all $B\in M(U,V,W)$ are Fredholm operators.
\end{tabular}

{\bf a) } For general $A\in M(U,V,W)$, the following implications
trivially hold:
\[
\begin{array}{ccccccc}
(i) & \Leftarrow & (ii) & \Leftarrow & (iii) & \Leftarrow\\
\Uparrow & & \Uparrow & & \Uparrow & &(iv)\\
(vii) & \Leftarrow & (vi) & \Leftarrow & (v) & \Leftarrow\\
\end{array}
\]

{\bf b) } If $A\in\PE(U,V,W)$ then $(i),\,(ii),\,(vi)$ and $(vii)$
are equivalent and $(iii)\iff (v)$ holds.

{\bf c) } If $A\in\PE(U,V,W)$ and $0\in U,W$ then $(i)-(vii)$ are
all equivalent.

{\bf d) } If $\delta>0$ in \eqref{eq:defminmax} then $(i)-(vii)$
hold, where all the inverses are bounded above by $1/\delta$.
\end{theorem}

\begin{remark} \label{rem:fullFSM}
{\bf a) } Statement b) of the theorem shows how `hard' it is for a
pseudoergodic operator to be Fredholm (i.e. invertible) or to even
have an applicable full FSM. It also shows that, like the
(essential) spectrum, these properties only depend on the sets
$U,V,W$ but not on the concrete operator $A\in\PE(U,V,W)$.

{\bf b) } If $U,V,W$ are discrete sets and $0\in U,W$ then
$A\in\PE(U,V,W)$ decouples into a block diagonal operator
$A=\Diag(B_i:i\in\Z)$ with every $B\in M_{\rm fin}(U,V,W)$ appearing
as one of the blocks $B_i$, so that some of the above claims in c)
become fairly obvious then. However, note that we do not assume
$U,V,W$ to be discrete in Theorem \ref{th:Fred_inv_stab}.

{\bf c) } The condition $\delta>0$ in d) is equivalent to
$V\subset\C\setminus (u^*+w^*)\ovD$ or, to phrase it in the style of
\eqref{eq:specA}, to $0\not\in V+(u^*+w^*)\ovD$.
\end{remark}

Since applicability of the full FSM of an operator $A\in\PE(U,V,W)$
is determined by the sets $U,V,W$ only rather than by the operator,
it seems advisable to use a version of the more flexible FSM
\eqref{eq:Anxn=b} that gives credit to individual features of the
concrete pseudoergodic operator $A$ and will work under the sole
condition of invertibility of $A$, where the full FSM might fail.


We say `might fail' because we do not have an example of sets
$U,V,W$ and $A\in\PE(U,V,W)$, where the full FSM fails while the
following version applies, unless when $\ind A_+\ne 0$. However, we
can prove that our adapted FSM from Section \ref{sec:aFSM} generally
applies if $A$ is invertible -- which we doubt in the case of the
full FSM (even if $\ind A_+=0$).

\subsection{The FSM with adaptive cut-off intervals}
\label{sec:aFSM}
If Theorem \ref{th:Fred_inv_stab} does not yield applicability of
the full FSM for $A\in\PE(U,V,W)$ then we propose using the FSM
\eqref{eq:Anxn=b} with cut-offs at integer values $(l_n)$ and
$(r_n)$ that are adapted to the operator $A$ at hand.
\medskip

{\bf The adaptive FSM in the general case. } We start with a
statement for general tridiagonal operators $A\in M(U,V,W)$ or, in
fact, for even more general operators. For $X=\ell^p(\Z)$ with
$p\in[1,\infty]$, we write $A\in BO(X)$ and call $A$ a {\sl band
operator} if $A$ acts via matrix-vector multiplication by a band
matrix. Moreover, we write $A\in BDO(X)$ and call $A$ a {\sl
band-dominated operator} if $A$ is the limit (in the operator norm
induced by $\|\cdot\|_X$) of a sequence of band operators.

If $A\in BDO(X)$ is given by the matrix $(a_{ij})_{i,j\in\Z}$ and
$B\in BDO(X)$ is given by a matrix $(b_{ij})_{i,j\in\Z}$ then we
call $B$ a {\sl limit operator} of $A$ if there exists a sequence
$h=(h_1,h_2,...)$ of integers with $|h_n|\to\infty$ and $a_{i+h_n,j+h_n}
\to b_{ij}$ as $n \to \infty$ for all $i,j\in\Z$. In this case we
write $B=:A_h$. For a given sequence $h$ of integers
going to infinity, let
\[
\opsp_h(A)\ :=\ \{A_g:g \textrm{ is an infinite subsequence of } h
\textrm{ for which } A_g \textrm{ exists}\}.
\]
By a Bolzano-Weierstrass argument it can be seen that $\opsp_h(A)$
is always nonempty. We will also abbreviate
\[
\opsp_+(A)\ :=\ \opsp_{(1,2,3,...)}(A)\qquad\textrm{and}\qquad
\opsp_-(A)\ :=\ \opsp_{(-1,-2,-3,...)}(A)
\]
and put $\opsp(A):=\opsp_+(A)\cup\opsp_-(A)$, so that the latter is
the set of all limit operators of $A$.

In the semi-infinite case $X=\ell^p(\N)$, the spaces $BO(X)$ and
$BDO(X)$ are defined in the same way. For $A\in BDO(X)$ one then
also defines limit operators $A_h$ exactly as above -- provided that
$h=(h_1,h_2,...)$ tends to $+\infty$. Note that $A_h$ is in any case
bi-infinite, i.e. it acts on $\ell^p(\Z)$.

The set of pseudoergodic operators can be equivalently characterized
in terms of limit operators (see \cite[\S 3.4.10]{Li:Book} or
\cite[\S 5.5.3]{Li:Habil}):

\begin{lemma} \label{lem:pe_limops}
For an operator $A\in M(U,V,W)$, one has
\[
A\ \in\ \left\{
\begin{array}{rcccl}
\PEL(U,V,W) &&\iff&& \opsp_-(A)\\
\PER(U,V,W) &&\iff&& \opsp_+(A)\\
\PE(U,V,W) &&\iff&& \opsp(A)
\end{array}
\right\}\ =\ M(U,V,W).
\]
\end{lemma}

Let $X=\ell^p(\Z)$ and $P:X\to X$ denote the operator of
multiplication by the characteristic function of $\N$, and let
$Q:=I-P$ be the complementary projector of $P$. Given an operator
$A\in BDO(X)$ with matrix $(a_{ij})_{i,j\in\Z}$, we write $A_+$ for
the compression $PAP|_{\im P}$ of $A$ to $\im P\cong\ell^p(\N)$;
that is, $A_+$ is the operator of multiplication by the matrix
$(a_{ij})_{i,j\in\N}$. Analogously, we write $A_-$ for the
compression $QAQ|_{\im Q}$ of $A$ to $\im
Q\cong\ell^p(\Z\setminus\N)$; that is, $A_-$ is the operator of
multiplication by the matrix $(a_{ij})_{i,j\in\Z\setminus\N}$. When
talking about their invertibility, Fredholmness or index, we always
understand $A_+$ and $A_-$ as operators on $\ell^p(\N)$, resp.
$\ell^p(\Z\setminus\N)$. If we are only interested in an operator on
$\ell^p(\N)$, we usually denote it by $A_+$ (indicating that it is
the compression of an operator $A$ on $\ell^p(\Z)$ to the positive
half-axis) to remind ourselves of the semi-infinite setting.

The following theorem is a generalization of results from
\cite{RaRoSi:FSMsubs,Li:FSMsubs} (which can be derived by
straightforward changes in the proofs there):

\begin{theorem} \label{th:FSM}
Let $X=\ell^p(\Z)$ with $p\in[1,\infty]$ and fix two sequences
$l=(l_n)_{n\in\N}$ and $r=(r_n)_{n\in\N}$ of integers
$l_1,l_2,...\to -\infty$ and $r_1,r_2,...\to +\infty$. For $A\in
BDO(X)$, the finite section method \eqref{eq:Anxn=b} is applicable
if and only if the following operators are invertible:
\begin{equation} \label{eq:limopsFSM}
A,\qquad \textrm{all operators }B_+\textrm{ with } B\in\opsp_l(A),
\qquad \textrm{all operators }C_-\textrm{ with } C\in\opsp_r(A).
\end{equation}
\end{theorem}

The set of operators \eqref{eq:limopsFSM} is particularly handy if
the sequences $l$ and $r$ are such that
\begin{equation} \label{eq:singletons}
\textrm{the sets } \{B_+ : B\in\opsp_l(A)\} \textrm{ and } \{C_- :
C\in\opsp_r(A)\} \textrm{ are singletons,}
\end{equation}
which is equivalent to the existence of the strong limits $B_+$ and
$C_-$ of
\[
\left(\begin{array}{cccc}
v_{l_n}&w_{l_n}\\
u_{l_n+1}&v_{l_n+1}&w_{l_n+1}\\[-2mm]
&u_{l_n+2}&v_{l_n+2}&\ddots\\[-0mm]
&&\ddots&\ddots
\end{array}\right)\qquad\textrm{and}\qquad
\left(\begin{array}{cccc} \ddots&\ddots\\[-2mm]
\ddots&v_{r_n-2}&w_{r_n-2}\\
&u_{r_n-1}&v_{r_n-1}&w_{r_n-1}\\
&&u_{r_n}&v_{r_n}
\end{array}\right)
\]
as $n\to\infty$.  For \eqref{eq:singletons} it is sufficient (but
not necessary) that the limit operators $A_l=:B$ and $A_r=:C$ exist.

%
%

Here is the version of Theorem \ref{th:FSM}
for semi-infinite matrices:
\begin{theorem} \label{th:FSM+}
Let $X=\ell^p(\N)$ with $p\in[1,\infty]$ and fix a monotonously
increasing sequence $r=(r_n)_{n\in\N}$ of positive integers. For
$A_+\in BDO(X)$, the finite section method \eqref{eq:Anxn=b}, with
$l_n=1$ for all $n\in\N$, is applicable if and only if the following
operators are invertible:
\begin{equation} \label{eq:limopsFSM+}
A_+,\qquad \textrm{all operators }C_-\textrm{ with }
C\in\opsp_r(A_+).
\end{equation}
\end{theorem}
Also here, the set \eqref{eq:limopsFSM+} is smallest possible if
$\{C_- : C\in\opsp_r(A_+)\}$ is a singleton.

%
%
\medskip

{\bf Bi-infinite bi-pseudoergodic systems. } We demonstrate how,
under the sole (and for this purpose minimal -- because necessary)
assumption of invertibility of $A$, one can approximately solve
operator equations $Ax=b$ on $\ell^p(\Z)$ with a bi-pseudoergodic
operator $A$ by the finite section method.

\begin{algorithm} \label{alg:bi} {\bf -- The $\PE_2$-FSM. }
Suppose $U,V,W\subset\C$ are non-empty and compact sets, $p\in
[1,\infty]$, $A\in\PE_2(U,V,W)$ is invertible and $b\in\ell^p(\Z)$
is given.

{\bf Step 1. } Pick some arbitrary $u\in U$, $v\in V$ and $w\in W$.
Choose integer sequences $l_1, l_2, ... $ monotonically decreasing
and $r_1,r_2,...$ monotonically increasing such that $l_n\le r_n$
and
\[
|u_i-u| + |v_i-v| + |w_i-w| < \frac 1n, \qquad \forall i\in\{l_n,
l_n+1,...,l_n+n\}\cup\{r_n-n,...,r_n-1,r_n\}
\]
for $n=1,2,...$ .

{\bf Step 2. } By Theorem \ref{th:Fred}, we know that we are in one
of the three cases (a), (b), (c). To find out which of these cases
applies, compute
\[
|v+2\sqrt{uw}|+|v-2\sqrt{uw}|-2(|u|+|w|)\ \left\{
\begin{array}{ll}
> 0 & \Rightarrow \textrm{ case } (a),\\
< 0 : & \textrm{compute } |u|-|w| \left\{
  \begin{array}{cl}
  > 0 & \Rightarrow  \textrm{ case } (b),\\
  < 0 & \Rightarrow  \textrm{ case } (c).
  \end{array} \right.
\end{array} \right.
\]
If one of the expressions to be computed here is zero or if the
outcome of this algorithm depends on the choice of $u,v,w$ then $A$
is not Fredholm, let alone invertible, by Theorem \ref{th:Fred}.

{\bf Step 3. } Depending on case (a), (b) or (c), we apply our
finite section method (with cut-offs at $(l_n)$ and $(r_n)$ as
chosen in step 1) to different equations:

In case (a), the FSM \eqref{eq:Anxn=b} is applicable to the equation
$Ax=b$, i.e. to
\[
\left(\begin{array}{ccccccc} \ddots&\ddots\\[-2mm]
\ddots&v_{-2}&w_{-2}\\
&u_{-1}&v_{-1}&w_{-1}\\
\hline &&u_{0}&v_0&w_0\\\hline
&&&u_{1}&v_1&w_1\\[-1mm]
&&&&u_2&v_2&\ddots\\
&&&&&\ddots&\ddots
\end{array}\right)
\left(\begin{array}{c} \vdots\\x(-2)\\x(-1)\\x(0)\\x(1)\\x(2)\\
\vdots\end{array}\right) \ =\ \left(\begin{array}{c} \vdots\\b(-2)\\b(-1)\\\hline b(0)\\ \hline b(1)\\b(2)\\
\vdots\end{array}\right).
\]

In case (b), the FSM \eqref{eq:Anxn=b} is applicable to the
following upward-translated (obviously equivalent) system
\[
\left(\begin{array}{ccccccc}
\ddots&\ddots&\ddots\\
&u_{-1}&v_{-1}&w_{-1}\\
\hline &&u_{0}&v_0&w_0\\ \hline
&&&u_{1}&v_1&w_1\\[-1.5mm]
&&&&u_2&v_2&\ddots\\[-1.5mm]
&&&&&u_3&\ddots\\[-1.5mm]
&&&&&&\ddots
\end{array}\right)
\left(\begin{array}{c} \vdots\\x(-2)\\x(-1)\\x(0)\\x(1)\\x(2)\\
\vdots\end{array}\right) \ =\ \left(\begin{array}{c} \vdots\\b(-1)\\ \hline b(0)\\ \hline b(1)\\b(2)\\
b(3)\\\vdots\end{array}\right).
\]

Finally, in case (c), the FSM \eqref{eq:Anxn=b} is applicable to the
following downward-translated (obviously equivalent) system
\[
\left(\begin{array}{ccccccc} \ddots\\[-2mm]
\ddots&w_{-3}\\[-2mm]
\ddots&v_{-2}&w_{-2}\\
&u_{-1}&v_{-1}&w_{-1}\\
\hline &&u_{0}&v_0&w_0\\ \hline
&&&u_{1}&v_1&w_1\\
&&&&\ddots&\ddots&\ddots
\end{array}\right)
\left(\begin{array}{c} \vdots\\x(-2)\\x(-1)\\x(0)\\x(1)\\x(2)\\
\vdots\end{array}\right) \ =\ \left(\begin{array}{c} \vdots\\b(-3)\\b(-2)\\b(-1)\\ \hline b(0)\\ \hline b(1)\\
\vdots\end{array}\right).
\]
\end{algorithm}

\begin{remark} {\bf --\, The growth of the intervals
$\{l_n,...,r_n\}$ as $n\to\infty$.} \label{rem:growth}

{\bf a) } If all entries $u_i$, $v_i$ and $w_i$ of $A$ are
independent samples from three uniformly distributed random
variables with values (everywhere) in $U$, $V$ and $W$ then, for all
choices $u,v,w$ in step 1, one expects the same exponential growth
of $-l_n$ and $r_n$. For example, if $U,V,W$ are finite with
$|U\times V\times W|=m$ then $-l_n$ and $r_n$ are of order $m^n$.
(In the random but not uniformly distributed case, it is certainly
advisable to pick some of the more likely $u,v,w$ in step 1 in order
to minimize the growth of $-l_n$ and $r_n$.)

{\bf b) } The choice of $l_n$ and $r_n$ in step 1 is such that the
conditions of Theorem \ref{th:FSM} and in particular condition
\eqref{eq:singletons} are met with operators $B_+$ and $C_-$ having
constant diagonals (containing $u,v$ and $w$). It is possible to aim
at different (non-Toeplitz) operators $B_+$ and $C_-$ here (as long
as they are invertible) via the choice of $l_n$ and $r_n$, while,
possibly, keeping the growth of $-l_n$ and $r_n$ more moderate.
Steps 2 (with arbitrarily picked $u,v,w$) and 3 still remain as
shown.

{\bf c) } One should not be too worried if these finite systems
become large very quickly since they can be solved in linear time
(as opposed to cubic, for the Gauss algorithm). In case (a) this is
done by the so-called Thomas algorithm \cite{Thomas}, while in case
(b), resp. (c), the solution is calculated successively via
backward, resp. forward, substitution.

{\bf d) } In \cite{RaRoSi:FSM_AP} the finite section method
\eqref{eq:Anxn=b} is adapted to the Almost Mathieu operator
\[
(Ax)_n\ =\ x_{n-1}\ + \lambda \cos\big(2\pi(n\alpha+\theta)\big)\,
x_n\ +\ x_{n+1},\qquad n\in\Z
\]
by putting $-l_n=r_n$ equal to the denominator of the $n$-th
continued fraction approximant of the irrational number $\alpha\in
(0,1)$. Note that this means $-l_n=r_n$ also grow exponentially in
$n$. For example, if $\alpha=(\sqrt 5-1)/2$ is the golden mean then
$-l_n=r_n$ is the $n$-th Fibonacci number. The order of exponential
growth is higher if the continued fraction expansion of $\alpha$
contains larger numbers. (For the golden mean, it is
$1/(1+1/(1+1/\cdots))$.)
\end{remark}
\medskip

{\bf Semi-infinite pseudoergodic systems. } For semi-infinite
systems $A_+x=b$ on $\ell^p(\N)$, the situation is related but much
simpler. Again, we only assume invertibility of the operator.

\begin{algorithm} \label{alg:semi} {\bf -- The $\PE_+$-FSM. }
Suppose $U,V,W\subset\C$ are non-empty and compact sets, $p\in
[1,\infty]$, $A_+\in\PE_+(U,V,W)$ is invertible and $b\in\ell^p(\N)$
is given.

{\bf Step 1. }  Pick some arbitrary $u\in U$, $v\in V$ and $w\in W$.
Choose a monotonically increasing sequence $r_1,r_2,...$ of positive
integers such that
\[
|u_i-u|\ +\ |v_i-v|\ +\ |w_i-w|\ <\ \frac 1n, \qquad \forall
i\in\{r_n-n,...,r_n-1,r_n\}
\]
for $n=1,2,...$ .

{\bf Step 2. } By Theorem \ref{th:Fred} and the invertibility of
$A_+$, we are automatically in case (a).

{\bf Step 3. } The FSM \eqref{eq:Anxn=b} with $l_n=1$ for all
$n\in\N$ and $(r_n)$ as chosen in step 1 applies to our equation
$A_+x=b$, i.e. to
\[
\left(\begin{array}{ccccc}
v_{1}&w_{1}\\
u_{2}&v_{2}&w_{2}\\
&u_{3}&v_3&w_3\\[-2mm]
&&u_{4}&v_4&\ddots\\
&&&\ddots&\ddots
\end{array}\right)
\left(\begin{array}{c} x(1)\\x(2)\\x(3)\\x(4)\\
\vdots\end{array}\right) \ =\ \left(\begin{array}{c}
b(1)\\b(2)\\b(3)\\b(4)\\ \vdots\end{array}\right).
\]
\end{algorithm}

As in Remark \ref{rem:growth} b), note that one could choose
$r=(r_n)$ in step 1 so that $C_-$, with $C\in\opsp_r(A)$, is not of
Toeplitz structure but is another (invertible) operator. For
example, one could choose $r_n$ such that
\[
|u_{r_n-i}-u_{i+2}|\ +\ |v_{r_n-i}-v_{i+1}|\ +\ |w_{r_n-i}-w_{i}|\
<\  \frac 1n, \qquad \forall i\in\{0,...,n\}
\]
for $n=1,2,...$, so that
\begin{equation} \label{eq:reflect}
\left(\begin{array}{cccc} \ddots&\ddots\\[-2mm]
\ddots&v_{r_n-2}&w_{r_n-2}\\
&u_{r_n-1}&v_{r_n-1}&w_{r_n-1}\\
&&u_{r_n}&v_{r_n}
\end{array}\right)\ \to\
\left(\begin{array}{cccc} \ddots&\ddots\\[-2mm]
\ddots&v_{3}&w_{2}\\
&u_{3}&v_{2}&w_{1}\\
&&u_{2}&v_{1}
\end{array}\right)\ =:\ C_-
\end{equation}
strongly as $n\to\infty$. But $C_-$ 
is invertible by our assumption on $A_+$:
%
We call the matrix $C_-$ in \eqref{eq:reflect} the {\sl reflection}
of the operator $A_+$ from \eqref{eq:A} and we will write $A_+^R$
for $C_-$. Conversely, we also call $A_+$ the reflection of $C_-$
and denote it by $C_-^R$. It is easy to see that a semi-infinite
matrix is invertible if and only if its reflection is invertible.

\subsection{Spectral and pseudospectral approximation}
Here we will briefly discuss another feature of the finite section
method with adaptive cut-off intervals. We will work exclusively on
the Hilbert space $H := \ell^2(\Z)$. Let again $l$ and $r$ be sequences
of negative and positive integers which converge to $-\infty$ and $+\infty$,
respectively. The set $\cA_{l,r}$ of all band-dominated operators $A$
on $H$ for which the limit operators $A_l$ and $A_r$ exist is a $C^*$-algebra,
as one easily checks. Note that every band-dominated operator $A$
belongs to an algebra $\cA_{l,r}$ with specified sequences $l,r$.
\begin{proposition} \label{fractal}
Let $A \in \cA_{l,r}$. Then the adaptive finite section
method (\ref{eq:Anxn=b}) with system matrices $P_{l_n,r_n}AP_{l_n,r_n}$
is fractal.
\end{proposition}
The notion of a {\em fractal approximation method} was introduced in
\cite{RoSi1}. Since already its definition makes heavily use of
$C^*$-algebraic language we will omit all technical details here and
refer the interested reader to \cite{RoSi1} and \cite{HaRoSi2}.
Roughly speaking, an algebra of approximation sequences is fractal
if every sequence in the algebra can be reconstructed from each if
its (infinite) subsequences modulo a sequence which tends to zero in
the norm. A single sequence like $(P_{l_n,r_n}AP_{l_n,r_n})_{n \in
\N}$ is called {\em fractal} if the smallest $C^*$-algebra which
contains this sequence and the sequence $(P_{l_n,r_n})_{n \in \N}$
has the fractal property. The proof of Proposition \ref{fractal}
follows easily from Theorem \ref{th:FSM} above and Theorem 1.69 in
\cite{HaRoSi2}. The main point is that the sequence
$(P_{l_n,r_n}AP_{l_n,r_n})_{n \in \N}$ is stable by Theorem
\ref{th:FSM} if and only if the operators $A$, $B_+ = PA_lP$ and
$C_- = QA_rQ$ are invertible and that the operators $A$, $B_+$ and
$C_-$ are strong limits of (shifts of) the sequence
$(P_{l_n,r_n}AP_{l_n,r_n})_{n \in \N}$. Since every subsequence has
the same strong limits, the result follows from Theorem 1.69 in
\cite{HaRoSi2}.

Fractal sequences are distinguished by their excellent convergence
properties. To mention only a few of them, let $\sigma (A)$ denote
the spectrum of an operator $A$, write $\sigma_2(A)$ for the set
of the singular values of $A$, i.e., $\sigma_2(A)$ is the set of
all non-negative square roots of elements in the spectrum of $A^*A$
and finally, for $\varepsilon > 0$, let $\sigma^{(\varepsilon)} (A)$
refer to the $\varepsilon$-pseudospectrum of $A$, i.e. to the set of all
$\lambda \in \C$ for which $A - \lambda I$ is not invertible or
$\|(A - \lambda I)^{-1}\| \ge 1/\varepsilon$. Let further
\[
d_H(M, \, N) := \max \, \{ \max_{m \in M} \min_{n \in N} |m-n|, \,
\max_{n \in N} \min_{m \in M} |m-n| \}
\]
denote the Hausdorff distance between the non-empty compact subsets
$M$ and $N$ of the complex plane.
\begin{theorem} \label{t95.41}
Let $A \in \cA_{l,r}$ and $A_n := P_{l_n,r_n}AP_{l_n,r_n}$. Then the
following sequences converge with respect to the Hausdorff distance
as $n \to \infty \!:$
\begin{itemize}
\item[$(a)$] $\sigma(A_n) \to \sigma (A) \cup \sigma (B_+) \cup
\sigma (C_-)$ if $A$ is self-adjoint;
\item[$(b)$] $\sigma_2 (A_n) \to \sigma_2 (A) \cup \sigma_2 (B_+) \cup
\sigma_2 (C_-)$;
\item[$(c)$] $\sigma^{(\varepsilon)} (A_n) \to \sigma^{(\varepsilon)}
(A) \cup \sigma^{(\varepsilon)} (B_+) \cup \sigma^{(\varepsilon)}
(C_-)$.
\end{itemize}
\end{theorem}
The proof follows immediately from the stability criterion in
Theorem \ref{th:FSM}, from the fractality of the sequence $(A_n)$ by
Proposition \ref{fractal}, and from Theorems 3.20, 3.23 and 3.33
in \cite{HaRoSi2}. Let us emphasize that in general one cannot
remove the assumption $A = A^*$ in assertion $(a)$, whereas $(c)$
holds without any assumption. This observation is only one
reason for the present increasing interest in pseudospectra.
For detailed presentations of pseudospectra and their applications
as well as of other spectral quantities see the monographs
\cite{BGr5,BSi2,HaRoSi2,TrE1} and the references therein.

\subsection{The selfadjoint case} \label{sec:selfadjoint}
We discuss briefly how our results simplify when $A$ is selfadjoint,
i.e. $w_i = \overline{u_{i+1}}$ for all $i$ in \eqref{eq:A}.

In that case, there are only two sets, $U$ and $V$, of which $V$ is
real. Instead of the ellipses $E(u,w)$, one looks at $E(u,\overline
u) = [ -2|u| , 2|u| ]$. The set $E(U,W)$ gets replaced by the union
of $E(u,\overline u)$ over all $u\in U$, which is simply $[ -2u^* ,
2u^* ]$.

In Theorem \ref{th:spec} a), the lower bound therefore becomes $V +
[ -2u^* , 2u^* ]$. But the upper bound becomes the same (recall
Example \ref{ex:spec} a) since $w^*=u^*$ and since all spectra are
real. So
\begin{equation} \label{eq:spec_selfadjoint}
\spec A\ =\ \spess A\ =\ \spess A_+\ =\ \spec A_+\ =\ V + [ -2u^* ,
2u^* ].
\end{equation}
In particular, $A$ is positive definite if and only if $\min_{v\in
V}>2u^*$.

In Theorem \ref{th:Fred} and anywhere else, case (a) applies (all
indices are zero of course). The theorem says that $V$ and $[ -2u^*
, 2u^* ]$ are disjoint if $A$ is Fredholm. From
\eqref{eq:spec_selfadjoint} we know that $V\cap[ -2u^* , 2u^*
]=\varnothing$ is indeed both necessary and sufficient for $A$ to be
Fredholm (i.e. invertible).

Concerning the FSM, not much simplification occurs apart from the
fact that the full FSM is applicable if $A$ is positive or negative
definite.


\section{Background theory and proofs}
Before we come to the deeper results, let us briefly show how
Fredholmness (and index) of $A$ is related to that of its half-axis
compressions $A_+$ and $A_-$. The following lemma is taken from
\cite{RaRoRoe,RaRoSi:IndexFSM}.

\begin{lemma} \label{lem:Fred}
An operator $A\in BO(X)$ is Fredholm if and only if both its
compressions $A_+$ and $A_-$ are Fredholm, i.e. $\spess A=\spess
A_+\cup\spess A_-$. Moreover, $\ind A=\ind A_++\ind A_-$.
\end{lemma}
\begin{proof}
Since $PAQ$ and $QAP$ are of finite rank if $A$ is a band operator,
one has that $A=PAP+PAQ+QAP+QAQ$ is equivalent, modulo compact
operators, to
\[
PAP+QAQ\ =\ (PAP+Q)(P+QAQ)\ =\ (P+QAQ)(PAP+Q).
\]
So $A$ is Fredholm if and only if $PAP+Q$ and $P+QAQ$, which are the
extensions (by identity) of $A_+$ and $A_-$ to $\ell^p(\Z)$, are
Fredholm. Moreover,
\begin{eqnarray*}
\ind A&=&\ind(PAP+PAQ+QAP+QAQ)\ =\ \ind(PAP+QAQ)\\ &=&
\ind(PAP+Q)(QAQ+P)\ =\ \ind(PAP+Q)+\ind(QAQ+P)\\ &=& \ind A_++\ind
A_-
\end{eqnarray*}
holds.
\end{proof}

One refers to $\ind A_-$ and $\ind A_+$ as the {\sl minus-} and the
{\sl plus-index} of $A$. By Lemma \ref{lem:Fred}, the problem of
determining Fredholmness (and the index) of $A$ splits into two
subproblems. These two subproblems again split into many smaller
problems, where the key notion is again that of a limit operator.
Besides Lemma \ref{lem:pe_limops} and Theorems \ref{th:FSM} and
\ref{th:FSM+}, limit operators feature in the following
characterization of Fredholmness (including the index):

\begin{theorem} \label{th:limops_Fred}
Let $X=\ell^p(\I)$ with $p\in[1,\infty]$ and
$\I\in\{\Z,\N,\Z\setminus\N\}$, and let $A\in BO(X)$.

{\bf a) } The following are equivalent

\begin{tabular}{rl}
  (i) & $A$ is Fredholm on $X$,\\
 (ii) & all limit operators of $A$ are invertible on $\ell^p(\Z)$ \cite{RaRoSi1998,RaRoSiBook},\\
(iii) & all limit operators of $A$ are injective on
$\ell^\infty(\Z)$ \cite{CWLi2008:FC,CWLi2008:Memoir},
\end{tabular}

so that, by applying the above to $A-\lambda I$ in place of $A$,
\begin{equation} \label{eq:spec_limops}
\spess A\ =\ \bigcup_{B\in\opsp(A)}\spec B\ =\
\bigcup_{B\in\opsp(A)}\sppt B.
\end{equation}

{\bf b) } If $A\in BO(\ell^p(\Z))$ is Fredholm then all operators in
$\opsp_-(A)$ have the same minus-index and all operators in
$\opsp_+(A)$ have the same plus-index, which also happen to be the
minus- and the plus-index of $A$, respectively
\cite{RaRoRoe,RaRoSi:IndexFSM}. This means $\ind A=\ind A_-+\ind
A_+$, where
\begin{eqnarray}
\label{eq:ind-}\ind A_-&=&\ind B_-\qquad\forall\ B\in\opsp_-(A),\\
\label{eq:ind+}\textrm{and}\qquad \ind A_+&=&\ind C_+\qquad\forall\
C\in\opsp_+(A).
\end{eqnarray}
\end{theorem}
Let $S$ denote the shift operator $(Sx)(m):=x(m-1)$, $m\in\Z$, on
$X=\ell^p(\Z)$. If $A$ is pseudoergodic then $\opsp(A)=M(U,V,W)$, by
Lemma \ref{lem:pe_limops}. Particularly simple elements of
$M(U,V,W)$ are operators whose matrix has constant diagonals. So fix
$u\in U$, $v\in V$ and $w\in W$, and let $L(u,v,w):=uS+vI+wS^{-1}$
be the single element of $M(\{u\},\{v\},\{w\})\subseteq M(U,V,W)$,
which is a so-called {\sl Laurent operator} (sometimes also called
``bi-infinite Toeplitz operator''). It is a standard result
\cite{BGr5,BSi2} that
\begin{equation} \label{eq:spec_Laurent}
\spec L(u,v,w)\ =\ \{ut^1+vt^0+wt^{-1}\,:\,t\in\T\}\ =\ v+E(u,w).
\end{equation}
Together with \eqref{eq:spec_limops}, the latter proves the lower
bound in Theorem \ref{th:spec}. The rather crude (but still helpful)
upper bound in Theorem \ref{th:spec} and statement d) in Theorem
\ref{th:Fred_inv_stab} rely on the following simple lemma and its
corollary.
\begin{lemma} \label{lem:perturb}
Let $A$ be a finite or (semi- or bi-)infinite matrix with
subdiagonal $\uu=(u_i)$, main diagonal $\vv=(v_i)$ and superdiagonal
$\ww=(w_i)$. Put
\begin{equation} \label{eq:defminmaxA}
u^*\,:=\,\sup_i|u_i|,\quad v_*\,:=\,\inf_i|v_i|,\quad
w^*\,:=\,\sup_i|w_i|,\quad \textrm{and}\quad
\delta_A\,:=\,v_*-(u^*+w^*).
\end{equation}
If $\delta_A>0$ then $A$ is invertible and $\|A^{-1}\|\le
1/\delta_A$.
\end{lemma}
\begin{proof}
Write $A=D+T$ with $D=\diag(v_i)$ and treat $A$ as a perturbation of
$D$. We have $A=D(I+D^{-1}T)$, where $D^{-1}T$ has subdiagonal
entries $u_iv_i^{-1}$, superdiagonal entries $w_iv_i^{-1}$ and
everything else zero. From $\delta_A>0$ we get that
\[
\|D^{-1}T\|\ \le\ \sup_i|u_iv_i^{-1}|\,+\,\sup_i|w_iv_i^{-1}|\ \le\
\frac{u^*+w^*}{v_*}\ <\ 1,
\]
so that $I+D^{-1}T$ is invertible by Neumann series. But from
$A^{-1}=(I+D^{-1}T)^{-1}D^{-1}$ it follows that also $A$ is
invertible and
\[
\|A^{-1}\|\ \le\ \|(I+D^{-1}T)^{-1}\|\,\|D^{-1}\|\ \le\ \frac
1{1-\|D^{-1}T\|}\,\|D^{-1}\|\ \le\ \frac 1{1-
\frac{u^*+w^*}{v_*}}\,\frac 1{v_*}\ =\ \frac 1{\delta_A},
\]
as was claimed.
\end{proof}

\begin{corollary} \label{cor:M(U,V,W)}
If $U,V,W\subset\C$ are non-empty and compact, \eqref{eq:defminmax}
holds with $\delta>0$, and if $A$ is in $M(U,V,W)$ or $M_+(U,V,W)$
or $M_{\rm fin}(U,V,W)$ then $A$ is invertible and $\|A^{-1}\|\le
1/\delta$.
\end{corollary}
\begin{proof}
Just note that $\delta_A\ge\delta$ for $\delta_A$ from
\eqref{eq:defminmaxA} and $\delta$ from \eqref{eq:defminmax} and
apply Lemma \ref{lem:perturb}.
\end{proof}
\medskip

Now we have all the machinery to prove our main results:

\begin{proofof}{Theorem \ref{th:spec}}
{\bf a) } All unions in this proof are taken over the set of all
$B\in M(U,V,W)$. By \eqref{eq:spec_limops} and Lemma
\ref{lem:pe_limops},
\[
\spess A\ =\ \cup\, \sppt B\ =\ \cup\, \spec B\ \supseteq\ \cup\,
\spess B\ \supseteq\ \spess A,
\]
so that equality holds in both ``$\supseteq$'' signs. Moreover,
\[
\spess A\ \subseteq\ \spec A\ \subseteq\ \cup\,\spec B\ =\ \spess A
\]
holds since $A$ is one of the operators $B$ in this union. So again
we have equality everywhere. Equality \eqref{eq:spec_limops} also
holds with $A$ replaced by $A_+$. Hence, by $\opsp(A_+)=M(U,V,W)$,
\[
\spec A\ \supseteq\ \spess A\ \supseteq\ \spess A_+\ =\ \cup\,\spec
B\ \supseteq\ \spec A
\]
holds, which proves the remaining equality in \eqref{eq:specA}. The
lower bound $V+E(U,W)$ in \eqref{eq:specA} now follows by evaluating
$\spec B$ from \eqref{eq:spec_Laurent} for all Laurent operators
$B=L(u,v,w)\in M(U,V,W)$. The upper bound $V+(u^*+w^*)\ovD$ follows
from Corollary \ref{cor:M(U,V,W)} since $A_+-\lambda I_+$ is
invertible if $|v-\lambda|>|u|+|w|$ for all $(u,v,w)\in U\times
V\times W$, i.e. if $\dist(\lambda,V)>u^*+w^*$.

{\bf b) } Let $A\in M(U,V,W)$, $u_*>w^*$ and suppose
$\lambda\in\cap_{v\in V}(v+(u_*-w^*)\D)$. Then
$|v-\lambda|<u_*-w^*\le |u|-|w|$ for all $u\in U$, $v\in V$ and
$w\in W$, so that the subdiagonal of $A-\lambda I$ dominates the
other two diagonals. By a simple perturbation argument as above (see
Lemma \ref{lem:perturb} and Corollary \ref{cor:M(U,V,W)}),
$S^{-1}(A-\lambda I)=\diag(u_i)(I+T)$ with $\|T\|<1$ is invertible,
and hence $A-\lambda I$ is invertible. The argument for the case
$w_*>u^*$ is completely symmetric.
\end{proofof}

\begin{proofof}{Theorem \ref{th:Fred}}
If $A$ is Fredholm then all its limit operators $B$, including the
Laurent operators $B:=L(u,v,w)\in M(U,V,W)$, are invertible. So, for
all $(u,v,w)\in U\times V\times W$, we have that $0\not\in\spec
L(u,v,w)=v+E(u,w)=v-E(u,w)$, i.e. $v\not\in E(u,w)$. The following
three cases are possible:

\begin{tabular}{cll}
(a) & $\wind(E(u,w),v)=0$,& i.e. $v$ is in the exterior of the ellipse $E(u,w)$,\quad or\\
(b) & $\wind(E(u,w),v)=1$, & i.e. $v$ is encircled counter-clockwise by $E(u,w)$,\quad or\\
(c) & $\wind(E(u,w),v)=-1$, & i.e. $v$ is encircled clockwise by $E(u,w)$,\\
\end{tabular}

\noindent where $\wind(C,z)$ denotes the winding number of a closed
oriented curve $C$ w.r.t. a point $z\not\in C$ and where the ellipse
$E(u,w)$ is parametrized (and thereby oriented) by the map
$\ph\mapsto ue^{\ri\ph}+we^{-\ri\ph}$ from $[0,2\pi)$ to $E(u,w)$. A
simple computation shows that $E(u,w)$ is oriented counter-clockwise
if $|u|>|w|$ and clockwise if $|u|<|w|$ (while the ellipse
degenerates into a line segment if $|u|=|w|$). Let $\varrho$ denote
the rotation $z\mapsto \frac v2 - z$ of the complex plane around
$\frac v2$. For the Toeplitz operator $B_+$, one has (see e.g.
\cite{BGr5,BSi2})
\begin{eqnarray*}
\ind B_+ &=& -\wind(\spec B,0)\ = \ -\wind(v+E(u,w),0)\ =\
-\wind(\varrho(v+E(u,w)),\varrho(0))\\
&=& \textstyle -\wind(-\frac v2 - E(u,w),\frac v2)\ =\
-\wind(-E(u,w),v) \ =\ -\wind(E(u,w),v),
\end{eqnarray*}
which is $0$ in case (a), $-1$ in case (b) and $1$ in case (c). By
\eqref{eq:ind+}, we have $\ind B_+=\ind A_+$ for all
$B\in\opsp_+(A)=M(U,V,W)$, so that for all choices $(u,v,w)\in
U\times V\times W$, the same case, (a), (b) or (c), applies --
according to $\ind A_+$.
\end{proofof}

\begin{proofof}{Theorem \ref{th:Fred_inv_stab}}
{\bf a) } The only implication that is not obvious here is that $A$
is invertible if the full FSM of $A$ is stable. This can be found in
\cite{RoSi} (also see \cite{RaRoSiBook,Li:Book}).

{\bf b) } Now let $A\in\PE(U,V,W)$. Properties $(i),\,(ii),\,(vi)$
and $(vii)$ are equivalent because $(i)$ implies $(vii)$ by Theorem
\ref{th:limops_Fred} a) and Lemma \ref{lem:pe_limops}. It remains to
show that $(iii)$ implies $(v)$. By Theorem \ref{th:FSM} with
$l=(-1,-2,...)$ and $r=(1,2,...)$, property $(iii)$ is equivalent to
invertibility of $A$ and all operators $B_+$ and $C_-$ with
$B\in\opsp_-(A)$ and $C\in\opsp_+(A)$. W.l.o.g suppose $A$ is
right-pseudoergodic, so that $\opsp_+(A)=M(U,V,W)$ by Lemma
\ref{lem:pe_limops} and hence $C_-$ is invertible for all $C\in
M(U,V,W)$. But then, for every $B\in M(U,V,W)$, also $B_+$ is
invertible because its reflection $B_+^R$ is of the form $C_-$ for
some $C\in M(U,V,W)$ and is therefore invertible. Finally, since
every operator in $M(U,V,W)$ is invertible if $A$ is invertible (see
above), we conclude $(v)$, by Theorem \ref{th:FSM} again.

{\bf c) } Now let $A\in\PE(U,V,W)$ and $0\in U,W$. To see that all
properties $(i)$--$(vii)$ are equivalent, it is sufficient to show
that $(i)$ implies $(iv)$. So let $A$ be Fredholm. By Theorem
\ref{th:limops_Fred} a) we know that all limit operators $B$ of $A$,
which are all operators in $M(U,V,W)$ by Lemma \ref{lem:pe_limops},
are invertible. Moreover, there is a $c>0$ (e.g. the norm of a
Fredholm regularizer of $A$, see \cite{RaRoSi1998}) such that
$\|B^{-1}\|\le c$ for all these operators $B$. Now we can show
$(iv)$: If $D\in M(U,V,W)$ then $B:=D\in M(U,V,W)$ is invertible and
$\|D^{-1}\|\le c$. If $D_+\in M_+(U,V,W)$ then $B:=\Diag(D_+^R,D_+)$
is in $M(U,V,W)$ since $0\in U,W$ and hence $B$ is invertible, so
that $D_+$ is invertible and $\|(D_+)^{-1}\|=\|B^{-1}\|\le c$.
Finally, if $D\in M_{\rm fin}(U,V,W)$ then, since $0\in U,W$,
$B:=\Diag(\cdots,D,D,D,\cdots)\in M(U,V,W)$ is invertible, so that
$D$ is invertible and $\|D^{-1}\|= \|B^{-1}\|\le c$.

{\bf d) } If $\delta>0$ then, by Corollary \ref{cor:M(U,V,W)},
property $(iv)$ holds and hence all the others follow.
\end{proofof}

\begin{proofof}{correctness of Algorithm \ref{alg:bi}}
The choice of the sequences $l=(l_n)$ and $r=(r_n)$ in step 1 is
such that the sets $\{D_+:D\in\opsp_l(A)\}=:\{B_+\}$ and
$\{D_-:D\in\opsp_r(A)\}=:\{C_-\}$ are singletons; in fact, $B_+$ and
$C_-$ are Toeplitz operators with diagonals $u$, $v$, $w$. It is
possible to choose sequences $l$ and $r$ with these properties
because $A\in\PE_2(U,V,W)$.

The test in step 2 exactly follows the geometric definition
\eqref{eq:ellipse} of the ellipse $E(u,w)$: If
$|v+2\sqrt{uw}|+|v-2\sqrt{uw}|>2(|u|+|w|)$ then $v\in E_-(u,w)$ and
we are in case (a) of Theorem \ref{th:Fred} (also see the proof of
Theorem \ref{th:Fred}). If
$|v+2\sqrt{uw}|+|v-2\sqrt{uw}|<2(|u|+|w|)$ then $v\in E_+(u,w)$ and
it remains to check the orientation of the ellipse. For $|u|>|w|$,
the ellipse is counter-clockwise oriented, so that case (b) applies,
and for $|u|<|w|$ the orientation is clockwise and we are in case
(c). By Theorem \ref{th:Fred}, the outcome of this test does not
depend on the values $u\in U$, $v\in V$, $w\in W$ chosen in step 1
if $A$ is Fredholm. The resulting case corresponds to $\ind A_+$.

If we are in case (a) then $\ind A_+=0$. Because $A$ is invertible,
we have $0=\ind A=\ind A_++\ind A_-=\ind A_-$. Now let
$D\in\opsp_r(A)$. By Theorem \ref{th:limops_Fred} a) and b), $D$ is
invertible and $\ind D_+=\ind A_+=0$, so that $0=\ind D=\ind
D_++\ind D_-=\ind A_++\ind C_-=\ind C_-$. So $C_-$ is a Toeplitz
operator that is Fredholm with index $0$. By Coburn's theorem
\cite{BGr5,BSi2}, $C_-$ is invertible. By a completely symmetric
argument (or simply by noting that $B_+=C_-^R$) we get that also
$B_+$ is invertible. Since \eqref{eq:singletons} holds, Theorem
\ref{th:FSM} yields the applicability of the FSM \eqref{eq:Anxn=b}
with the sequences $l$ and $r$ as chosen in step 1.

If we are in case (b) or (c) then $k:=\ind A_+=\mp 1$ and there are
no sequences $l$ and $r$ for which the FSM \eqref{eq:Anxn=b} could
be applicable to $Ax=b$. However, the FSM \eqref{eq:Anxn=b} is
applicable to the equivalent system $S^kAx=S^kb$ by the same
arguments as in case (a) since $S^kA$ is invertible and since
$\ind(S^kA)_+=\ind S^k_+A_+=\ind S^k_++\ind A_+=-k+k=0$.
\end{proofof}

\begin{proofof}{correctness of Algorithm \ref{alg:semi}}
The choice of the sequence $r=(r_n)$ in step 1 is such that the set
$\{D_-:D\in\opsp_r(A)\}=:\{C_-\}$ is a singleton; in fact, $C_-$ is
a Toeplitz operator with diagonals $u$, $v$, $w$. It is possible to
choose such a sequence $r$ because $A_+\in\PE_+(U,V,W)$.

Since $\ind A_+=0$ by invertibility of $A_+$, we are automatically
in case (a) of Theorem \ref{th:Fred}. Let $D\in\opsp_r(A)$. By
Theorem \ref{th:limops_Fred} a) and b), $D$ is invertible and $\ind
D_+=\ind A_+=0$, so that $0=\ind D=\ind D_++\ind D_-=\ind A_++\ind
C_-=\ind C_-$. So $C_-$ is a Toeplitz operator that is Fredholm with
index $0$. By Coburn's theorem \cite{BGr5,BSi2}, $C_-$ is
invertible. Now Theorem \ref{th:FSM+} yields the applicability of
the FSM \eqref{eq:Anxn=b} with the sequence $r$ as chosen in step 1.
\end{proofof}
\medskip

\section{A numerical example}
We illustrate our results by a numerical computation, for which we
come back to the Hatano-Nelson model from Example \ref{ex:spec} b).
So let $U=\{e^g\}$ and $W=\{e^{-g}\}$ with $g>0$, put
$c:=e^g+e^{-g}=2\cosh g$ and $s:=e^g-e^{-g}=2\sinh g$, and let
$V=[-a,a]$ with $0<a<s<c$.

Now let $A\in\PE(U,V,W)$. From Theorem \ref{th:spec} and our
discussion in Example \ref{ex:spec} b), we derive the upper and
lower bounds on $\spec A$ as shown in Figure \ref{fig:HatNel}. For
further studies of this operator, including the size and shape of
the hole in its spectrum, see
\cite{Davies2001:SpecNSA,Davies2001:PseudoErg,Davies2005:HigherNumRanges,MartinezThesis,MartinezHN}.

\noindent
\begin{center}
\includegraphics[width=\textwidth]{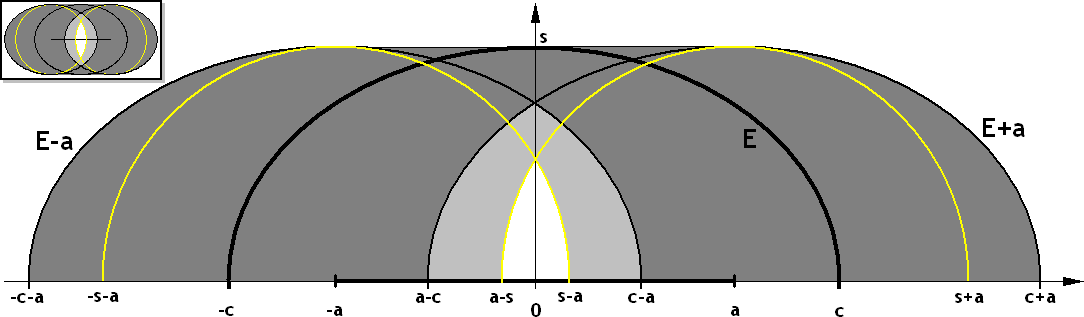}
\end{center}
~\\[-18mm]  
\begin{figure}[h]
\caption{\footnotesize Here are our lower (dark gray) and upper
(dark+light gray) bound on $\spec A$ from Example \ref{ex:spec} b)
with $g=1$ and $a=2$, so that $0<a<s<c$. The region of uncertainty
(light gray) is small if $s$ and $c$ are close to each other (i.e.
if $g$ is large). For reasons of symmetry we have only shown the
upper half of the complex plane.} \label{fig:HatNel}
\end{figure}

While the spectrum of a $n$-by-$n$ principal submatrix of $A$ is
less interesting (each such matrix is similar to a self-adjoint
matrix), the spectrum of this finite problem with periodic boundary
conditions and its limit as $n\to\infty$ has been described in much
detail by Goldsheid and Khoruzhenko \cite{GoldKoru}, who thereby
verified numerical observations of Hatano and Nelson
\cite{HatanoNelson96,HatanoNelson97,HatanoNelson98}. The limiting
set as $n\to\infty$ turns out to be the union of certain analytic
curves (the so-called `bubble with wings' \cite{TrefContEmb}) and is
entirely different from (although contained in) the spectrum of the
infinite matrix $A$.

We will now apply our adaptive FSM (Algorithm \ref{alg:bi}) to a
concrete matrix $A\in M(U,V,W)$, whose main diagonal entries $v_i$
have been chosen independently from $V=[-a,a]$, where the density of
our probability distribution on $V$ increases in a certain way
towards the endpoints of the interval. Our model has the parameters
$g=1$ (so that $c=2\cosh 1\approx 3.0862$ and $s=2\sinh 1\approx
2.3504$) and $a=2<s$, whence $0\not\in\spec A$.

In step 1 of the algorithm, we choose, as motivated in Remark
\ref{rem:growth} a), $v=2$, which is one of the values with the
highest probability density, besides the obvious choices $u=e^1$ and
$w=e^{-1}$. Then, for $n=1,2,...$, we look for $n$ consecutive
entries of the main diagonal that are within $1/n$ of $v=2$ to find
our cut-off bounds $l_n<l_{n-1}$ and $r_n>r_{n-1}$ (see Figure
\ref{fig:v}) with $l_0=0=r_0$.

\noindent
\begin{center}
\includegraphics[width=\textwidth]{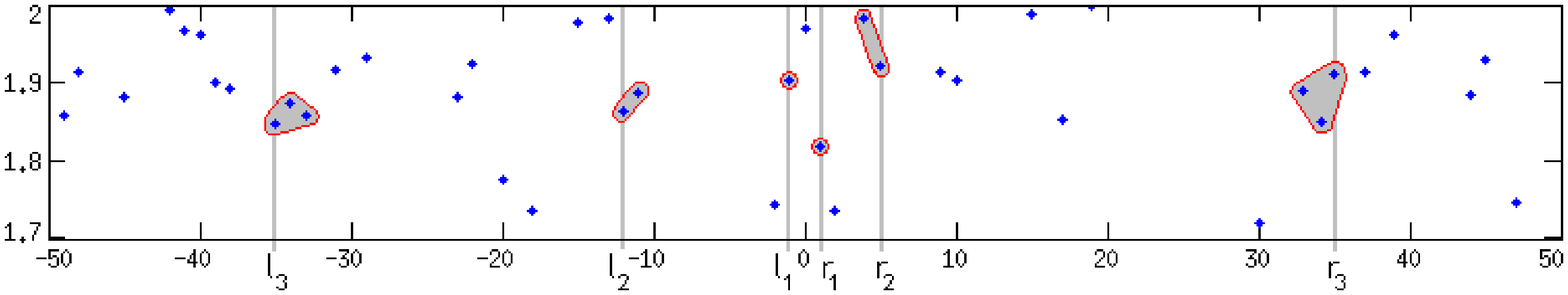}
\end{center}
~\\[-18mm]  
\begin{figure}[h]
\caption{\footnotesize These are the main diagonal entries $v_{-50}$
to $v_{50}$ that are close to $2$. Encircled are the groups of
$n=1,2,3$ consecutive entries that are within $1/n$ of $v=2$ and
therefore lead to the definition of $l_n$ and $r_n$.} \label{fig:v}
\end{figure}

In step 2 of the algorithm, we find that we are in case (b), which
says that $v$ lies inside the ellipse $E=E(u,w)$ and it is encircled
counter-clockwise w.r.t. the parametrization $\ph\mapsto
ue^{\ri\ph}+we^{-\ri\ph}=e^{1+\ri\ph}+e^{-1-\ri\ph}$ of $E$. In
other words: $\ind A_+=-1$.

This means that, in step 3, we shift our infinite system up by one
row before we truncate it according to our sequences $l=(l_n)$ and
$r=(r_n)$. The resulting method is applicable if and only if the
inverses of the finite matrices $A_n:=P_{l_n,r_n}S^{-1}AP_{l_n,r_n}$
remain uniformly bounded as $n\to\infty$. The following table shows
the cut-off sequences $l=(l_n)$ and $r=(r_n)$, the size of the
matrices $A_n$ and the norms of their inverses for $n=1,2,...,8$.
\[
\begin{array}{|r||r|r|r|r|}
\hline n&l_n&r_n&r_n-l_n+1&\|A_n^{-1}\|\\
\hline 1&-1&1&3&0.6816\\
2&-12&5&18&1.0580\\
3&-35&35&71&1.2698\\
4&-41&162&204&1.2698\\
5&-899&537&1437&1.4121\\
6&-1068&1183&2252&1.5438\\
7&-20494&21758&42253&1.6135\\
8&-469241&41570&510811&1.7500\\ \hline
\end{array}
\]
We see the rather irregular exponential growth of the intervals
$\{l_n,...,r_n\}$ (see Remark \ref{rem:growth}) and the moderate
growth of the inverses $A_n^{-1}$. This numerical evidence is not
really convincing that the inverses remain uniformly bounded as
$n\to\infty$. However, from the theory behind our Theorem
\ref{th:FSM} it follows that $\limsup_n \|A_n^{-1}\|$ is in case
$p=2$ equal (and otherwise at least bounded above by two times) the
maximum of the norms of the inverses of the operators in
\eqref{eq:limopsFSM}.
%
%
In our case, this means that
\[
\limsup_n \|A_n^{-1}\|\ =\ \max(\,\|A^{-1}\|\,,\,\|B_+^{-1}\|\,)
\]
if $p=2$, where $B_+$ is the Toeplitz operator (note the translation
$S^{-1}$ in step 3)
\[
B_+\ =\ \left(\begin{array}{cccc} e^1&2&e^{-1}\\
&e^1&2&\ddots\\
&&e^1&\ddots\\
&&&\ddots
\end{array}\right)
\]
with symbol $a(t)=e^1+2t^{-1}+e^{-1}t^{-2}$, $t\in\T$. But from
$\|A^{-1}\|\le (s-2)^{-1}\approx 2.8539$ (the argument is as in the
proof of Lemma \ref{lem:perturb}) and
$\|B_+^{-1}\|=(\min_{t\in\T}|a(t)|)^{-1}=(c-2)^{-1}\approx 0.9207$
(note that $B_+$ is upper-triangular, whence its inverse is the
Toeplitz operator with symbol $a(t)^{-1}$), we get that $\limsup_n
\|A_n^{-1}\|$ is bounded above by $(s-2)^{-1}\approx 2.8539$. In
general, it takes very large random matrices $A_n$ to see
$\|A_n^{-1}\|$ come close to $\sup_n \|A_n^{-1}\|$ because this
requires a particular (and usually long) pattern somewhere on the
diagonal(s) of $A_n$. The latter is reminiscent of the finite but
very long time it takes a monkey to type the complete works of
Shakespeare \cite{Monkey}.
\medskip

{\bf Acknowledgements. } The first author acknowledges the financial
support by Marie-Curie Grants MEIF-CT-2005-009758 and
PERG02-GA-2007-224761 of the EU.

\bigskip

\noindent {\bf Authors:}\\[4mm]
Marko Lindner\hfill \href{mailto:marko.lindner@mathematik.tu-chemnitz.de}{{\tt marko.lindner@mathematik.tu-chemnitz.de}}\\
TU Chemnitz\\
Fakult\"at Mathematik\\
D-09107 Chemnitz\\
GERMANY\\[8mm]
Steffen Roch\hfill \href{mailto:roch@mathematik.tu-darmstadt.de}{{\tt roch@mathematik.tu-darmstadt.de}}\\
TU Darmstadt\\
Fachbereich Mathematik\\
Schlossgartenstr. 7\\
D-64289 Darmstadt\\
GERMANY


\begin{thebibliography}{99} 
%
\bibitem{Anderson58}
{\sc P.~W.~Anderson:} Absence of diffusion in certain random
lattices, {\it Phys. Rev.} {\bf 109} (1958), 1492--1505.

\bibitem{Anderson61}
{\sc P.~W.~Anderson:} Localized Magnetic States in Metals, {\it
Phys. Rev.} {\bf 124} (1961), 41-–53.

\bibitem{Baxter}
{\sc G.~Baxter:} A norm inequality for a 'finite-section'
Wiener-Hopf equation, {\it Illinois J. Math.}, 1962, 97--103.

\bibitem{BGr5}
{\sc A.~B\"ottcher} and {\sc S.~M.~Grudsky:} {\it Spectral
Properties of Banded Toeplitz Matrices}, siam, Philadelphia 2005.


\bibitem{BSi2}
{\sc A.~B\"ottcher} and {\sc B.~Silbermann:} {\it Introduction to
Large Truncated Toep\-litz Matrices}, Springer, Berlin, Heidelberg
1999.


\bibitem{CW.Heng.ML:Sierp}
{\sc S.~N.~Chandler-Wilde, R.~Chonchaiya} and {\sc M.~Lindner}:
Eigenvalue problem meets Sierpinski triangle: Computing the spectrum
of a non-selfadjoint random operator, to appear in {\it Operators
and Matrices}.

\bibitem{CW.Heng.ML:UpperBounds}
{\sc S.~N.~Chandler-Wilde, R.~Chonchaiya} and {\sc M.~Lindner:}
Upper Bounds on the Spectra and Pseudospectra of Jacobi and Related
Operators, {\it in preparation}.

\bibitem{CW.Heng.ML:tridiag}
{\sc S.~N.~Chandler-Wilde, R.~Chonchaiya} and {\sc M.~Lindner:} On
the Spectra and Pseudospectra of a Class of non-self-adjoint Random
Matrices and Operators, {\it in preparation}.

\bibitem{CWLi2008:FC}
{\sc S.~N.~Chandler-Wilde} and {\sc M.~Lindner}: Sufficiency of
Favard's condition for a class of band-dominated operators on the
axis, {\it J. Funct. Anal.} {\bf 254} (2008), 1146--1159.

\bibitem{CWLi2008:Memoir}
{\sc S.~N.~Chandler-Wilde} and {\sc M.~Lindner}: {\it Limit
Operators, Collective Compactness, and the Spectral Theory of
Infinite Matrices}, Memoirs of the AMS, Vol. 210, Nr. 989, 2011.

\bibitem{Thomas}
{\sc S.~D.~Conte} and {\sc C.~deBoor}: {\it Elementary Numerical
Analysis}, McGraw-Hill, New York, 1972.

\bibitem{Davies2001:SpecNSA}
{\sc E.~B.~Davies:} Spectral properties of non-self-adjoint matrices
and operators, {\it Proc. Royal Soc. A.} {\bf 457} (2001), 191--206.

\bibitem{Davies2001:PseudoErg}
{\sc E.~B.~Davies:} Spectral theory of pseudo-ergodic operators,
{\it Commun. Math. Phys.} {\bf 216} (2001), 687--704.

\bibitem{Davies2005:HigherNumRanges}
{\sc E.~B.~Davies:} Spectral bounds using higher order numerical
ranges, {\it LMS Journal of Computation and Mathematics}, {\bf 8}
(2005), 17--45.

\bibitem{Davies2007:Book}
{\sc E.~B.~Davies:} {\it Linear Operators and their Spectra},
Cambridge University Press, 2007.



\bibitem{FeinZee99a}
{\sc J.~Feinberg} and {\sc A.~Zee:} Non-Hermitean Localization and
De-Localization, {\it Phys. Rev. E} {\bf 59} (1999), 6433--6443.

\bibitem{FeinZee99b}
{\sc J.~Feinberg} and {\sc A.~Zee:} Spectral Curves of Non-Hermitean
Hamiltonians, {\it Nucl. Phys. B} {\bf 552} (1999), 599--623.

\bibitem{GohbergFeldman}
{\sc I.~Gohberg} and {\sc I.~A.~Feldman}, {\it Convolution equations
and projection methods for their solution}, Transl. of Math.
Monographs, {\bf 41}, Amer. Math. Soc., Providence, R.I., 1974
[Russian original: Nauka, Moscow, 1971].

\bibitem{GoldKoru}
{\sc I.~Goldsheid} and {\sc B.~Khoruzhenko:} Eigenvalue curves of
asymmetric tridiagonal random matrices, Electronic Journal of
Probability {\bf 5} (2000), 1--28.

\bibitem{HaRoSi2}
{\sc R.~Hagen, S.~Roch} and {\sc B.~Silbermann:} {\it $C^*-$Algebras
and Numerical Analysis}, Marcel Dekker, Inc., New York, Basel, 2001.

\bibitem{HatanoNelson96}
{\sc N.~Hatano} and {\sc D.~R.~Nelson:} Localization transitions in
non-Hermitian quantum mechanics, {\it Phys. Rev. Lett.} {\bf 77}
(1996), 570--573.

\bibitem{HatanoNelson97}
{\sc N.~Hatano} and {\sc D.~R.~Nelson:} Vortex Pinning and
Non-Hermitian Quantum Mechanics, {\it Phys. Rev. B} {\bf 56} (1997),
8651--8673.

\bibitem{HatanoNelson98}
{\sc N.~Hatano} and {\sc D.~R.~Nelson:} Non-Hermitian Delocalization
and Eigenfunctions, {\it Phys. Rev. B} {\bf 58} (1998), 8384--8390.

\bibitem{HolzOrlandZee}
{\sc D.E.~Holz, H.~Orland} and {\sc A.~Zee}: On the remarkable
spectrum of a non-Hermitian random matrix model, {\it Journal of
Physics A: Mathematical and General} {\bf 36} (2003), 3385--3400.

\bibitem{Kozak}
{\sc A.~V.~Kozak:} A local principle in the theory of projection
methods, {\it Dokl. Akad. Nauk SSSR} {\bf 212} (1973), 1287–1289;
English transl. {\it Soviet Math. Dokl.} {\bf 14} (1973).

\bibitem{KozakSimonenko}
{\sc A.~V.~Kozak} and {\sc I.~V.~Simonenko:} Projectional methods
for solving multidimensional discrete equations in convolutions,
{\it Sib. Mat. Zh.} {\bf 21} (1980), 119--127.

\bibitem{Kurbatov}
{\sc V.~G.~Kurbatov:} {\it Functional Differential Operators and
Equations}, Kluwer Academic Publishers, Dordrecht, Boston, London
1999.

\bibitem{Li:Book}
{\sc M.~Lindner}: {\it Infinite Matrices and their Finite Sections:
An Introduction to the Limit Operator Method}, Frontiers in
Mathematics, Birkh\"auser 2006.

\bibitem{Li:Wiener}
{\sc M.~Lindner:} Fredholmness and index of operators in the Wiener
algebra are independent of the underlying space, {\it Operators and
Matrices} {\bf 2} (2008), 297--306.

\bibitem{Li:Habil}
{\sc M.~Lindner:} Fredholm Theory and Stable Approximation of Band
Operators and Generalisations, {\it Habilitation thesis, TU
Chemnitz}, 2009.

\bibitem{Li:BiDiag}
{\sc M.~Lindner:} A note on the spectrum of bi-infinite bi-diagonal
random matrices, {\it Journal of Analysis and Applications} {\bf 7}
(2009), 269--278.

\bibitem{Li:FSMsubs}
{\sc M.~Lindner:} The finite section method and stable subsequences,
{\it Applied Numerical Mathematics} {\bf 60} (2010), 501--512.

\bibitem{MartinezThesis}
{\sc C.~Mart\'inez:} Spectral Properties of Tridiagonal Operators,
{\it PhD thesis, Kings College, London} 2005.

\bibitem{MartinezHN}
{\sc C.~Mart\'inez:} Spectral estimates for the one-dimensional
non-self-adjoint Anderson model, {\it J. Operator Theory}, {\bf 56}
(2006), 59--88.

\bibitem{NelsonShnerb}
{\sc D.R.~Nelson} and {\sc N.M.~Shnerb}: Non-Hermitian localization
and population biology, {\it Phys. Rev. E} {\bf 58} (1998),
1383--1403.



\bibitem{RaRoRoe}
{\sc V.~S.~Rabinovich, S.~Roch} and {\sc J.~Roe:} Fredholm indices
of band-dominated operators, {\it Integral Equations Operator
Theory} {\bf 49} (2004), no. 2, 221--238.

\bibitem{RaRoSi1998}
{\sc V.~S.~Rabinovich, S.~Roch} and {\sc B.~Silbermann:} Fredholm
Theory and Finite Section Method for Band-dominated operators, {\it
Integral Equations Operator Theory} {\bf 30} (1998), no. 4,
452--495.

\bibitem{RaRoSi2001}
{\sc V.~S.~Rabinovich, S.~Roch} and {\sc B.~Silbermann:}
Band-dominated operators with operator-valued coefficients, their
Fredholm properties and finite sections, {\it Integral Equations
Operator Theory} {\bf 40} (2001), no. 3, 342--381.

\bibitem{RaRoSi2001:FSM}
{\sc V.~S.~Rabinovich, S.~Roch} and {\sc B.~Silbermann:} Algebras of
approximation sequences: Finite sections of band-dominated
operators, {\it Acta Appl. Math.} {\bf 65} (2001), 315--332.

\bibitem{RaRoSiBook}
{\sc V.~S.~Rabinovich, S.~Roch} and {\sc B.~Silbermann:} {\it Limit
Operators and Their Applications in Operator Theory}, Birkh\"auser
2004.

\bibitem{RaRoSi:FSM_AP}
{\sc V.~S.~Rabinovich, S.~Roch} and {\sc B.~Silbermann:} Finite
sections of band-dominated operators with almost periodic
coefficients, {\it Operator Theory: Advances and Applications} {\bf
170} (2007), 205--228.

\bibitem{RaRoSi:FSMsubs}
{\sc V.~S.~Rabinovich, S.~Roch} and {\sc B.~Silbermann:} On finite
sections of band-dominated operators, {\it Operator Theory: Advances
and Applications} {\bf 181} (2008), 385--391.

\bibitem{RaRoSi:IndexFSM}
{\sc V.~S.~Rabinovich, S.~Roch} and {\sc B.~Silbermann:} The finite
sections approach to the index formula for band-dominated operators,
{\it Operator Theory: Advances and Applications} {\bf 187} (2008),
185--193.

\bibitem{Roch:ellp}
{\sc S.~Roch:} Band-dominated operators on $\ell^p-$spaces: Fredholm
indices and finite sections, {\it Acta Sci. Math.} {\bf 70} (2004),
no. 3--4, 783--797.

\bibitem{Roch:FSM}
{\sc S.~Roch:} {\it Finite sections of band-dominated operators},
Memoirs of the AMS, Vol. 191, Nr. 895, 2008.



\bibitem{RoSi}
{\sc S.~Roch} and {\sc B.~Silbermann:} Non-strongly converging
approximation methods, {\it Demonstratio Math.} {\bf 22} (1989), no.
3, 651--676.

\bibitem{RoSi1}
{\sc S.~Roch} and {\sc B.~Silbermann:} $C^*$-Algebra techniques in
numerical analysis, {\it J. Oper. Theory} {\bf 35} (1996), no. 2, 241--280.

\bibitem{SeidelSilbermann1}
{\sc M.~Seidel} and {\sc B.~Silbermann}: Finite Sections of
Band-Dominated Operators: $l^p$-Theory, {\it Complex Analysis and
Operator Theory}, {\bf 2} (2008), 683--699.

\bibitem{SeidelSilbermann2}
{\sc M.~Seidel} and {\sc B.~Silbermann}: Banach Algebras of
Structured Matrix Sequences, {\it Linear Algebra and Applications}
{\bf 430} (2009), 1243--1281.


\bibitem{TrefContEmb}
{\sc L.~N.~Trefethen, M.~Contedini} and {\sc M.~Embree:} Spectra,
pseudospectra, and localization for random bidiagonal matrices, {\it
Comm. Pure Appl. Math.} {\bf 54} (2001),
595--623.

\bibitem{TrE1}
{\sc L.~N.~Trefethen} and {\sc M.~Embree:} {\it Spectra and Pseudospectra:
the Behavior of Nonnormal Matrices and Operators}, Princeton Univ. Press,
Princeton, NJ, 2005.

\bibitem{Monkey}
Infinite monkey theorem, \href{http://en.wikipedia.org/wiki/Infinite_monkey_theorem}{{\tt
http://en.wikipedia.org/wiki/Infinite\_monkey\_theorem}}

\end{thebibliography}
\end{document}